\documentclass[12pt, a4paper]{article}
\usepackage{amssymb, amsmath,amsthm }
\usepackage[all]{xy}
\newtheorem{prop}{Proposition}[section]
\newtheorem{defi}[prop]{Definition}

\newtheorem{lem}[prop]{Lemma}
\newtheorem{thm}[prop]{Theorem}

\newtheorem{remar}[prop]{Remark}
\newtheorem{cor}[prop]{Corollary}

\DeclareMathAlphabet{\mathpzc}{OT1}{pzc}{m}{it}
\DeclareMathOperator{\Aut}{Aut}
\DeclareMathOperator{\End}{End}
\DeclareMathOperator{\Hom}{Hom}
\DeclareMathOperator{\Ind}{Ind}
\DeclareMathOperator{\cInd}{c-Ind}

\DeclareMathOperator{\Sym}{Sym}

\DeclareMathOperator{\GL}{GL}
\DeclareMathOperator{\SL}{SL}
\DeclareMathOperator{\Ker}{Ker}

\DeclareMathOperator{\Gal}{Gal}

\DeclareMathOperator{\soc}{soc}

\DeclareMathOperator{\Irr}{Irr}

\DeclareMathOperator{\supp}{Supp}
\DeclareMathOperator{\id}{id}
\DeclareMathOperator{\Mod}{Mod}

\DeclareMathOperator{\val}{val}
\DeclareMathOperator{\Sp}{Sp}
\DeclareMathOperator{\Rep}{Rep}

\DeclareMathOperator{\St}{St}

\DeclareMathOperator{\dt}{det}
\DeclareMathOperator{\Ban}{Ban}
\DeclareMathOperator{\ind}{ind}
\DeclareMathOperator{\rad}{rad}
\newcommand{\cIndu}[3]{\cInd_{#1}^{#2}{#3}}

\newcommand{\Indu}[3]{\Ind_{#1}^{#2}{#3}}
\newcommand{\oF}{\mathfrak o}
\newcommand{\pF}{\mathfrak{p}}

\newcommand{\pif}{\varpi}
\newcommand{\Qp}{\mathbb {Q}_p}
\newcommand{\Zp}{\mathbb{Z}_p}
\newcommand{\Qpbar}{\overline{\mathbb{Q}}_p}

\newcommand{\Qbar}{\overline{\mathbb{Q}}_p}

\newcommand{\GG}{\mathcal G}
\newcommand{\HH}{\mathcal H}

\newcommand{\Eins}{\mathbf 1}
\newcommand{\PP}{\mathcal P}
\newcommand{\MM}{\mathfrak M}

\newcommand{\KK}{\mathfrak K}
\newcommand{\ZZ}{\mathbb Z}

\newcommand{\CC}{\mathfrak C}
\newcommand{\DD}{\mathfrak D}
\newcommand{\mfg}[1]{\Mod_{fg}(#1)}
\newcommand{\mfgfl}[1]{\Mod_{fg}^{fl}(#1)}
\newcommand{\mfl}{\Mod^{fl}_{comp}}
\newcommand{\QQ}{\mathbb Q}
\newcommand{\Aa}{\mathfrak A}

\newcommand{\ab}{\mathcal A}
\newcommand{\Fp}{\mathbb F_p}
\newcommand{\Fq}{\mathbb F_q}
\newcommand{\Fbar}{\overline{\mathbb F}_p}
\newcommand{\II}{\mathcal I}
\newcommand{\RR}{\mathfrak R}
\newcommand{\RRs}{\RR_{\mathfrak s}} 
\title{Admissible unitary completions of locally $\Qp$-rational representations of $\GL_2(F)$}
\author{Vytautas Pa\v{s}k\={u}nas}
\date{\today.}
\begin{document} 
\maketitle
\begin{abstract}Let $F$ be a finite extension of $\Qp$, $p>2$. We construct admissible unitary completions 
of certain representations of $\GL_2(F)$ on $L$-vector spaces, where $L$ is a finite extension of $F$. When 
$F=\Qp$ using the results of Berger, Breuil and Colmez we obtain some results about lifting $2$-dimensional 
mod $p$ representations of the absolute Galois group of $\Qp$ to  crystabelline representations with given Hodge-Tate
weights.
\end{abstract}

\tableofcontents

\section{Introduction}
Let $F$ be a finite extension of $\Qp$ with the ring of integers $\oF$, uniformizer $\pif$ and the residue field isomorphic to $\Fq$. Let $G:=\GL_2(F)$ and $K:=\GL_2(\oF)$.
Let $L$ be a 'large' finite extension 
of $\Qp$, ring of integers $A$, $\MM$ the maximal ideal in $A$ and residue field $k=k_L$.  
Let $\mathfrak R$ be the category of smooth representations of $G$ on $\overline{L}$-vector spaces, 
then $\RR$ decomposes into a product of 
subcategories $\RR\cong \prod_{\mathfrak s\in \mathfrak B} \RRs,$
where $\mathfrak B$  is the set of inertial equivalence classes of supercuspidal representations of the Levy subgroups of $G$, see \cite{ber}, \cite{bk}. Following
Henniart \cite[Def A.1.4.1]{henniart} we say that an irreducible smooth 
$\overline{L}$-representation $\tau$ of $K$ is \textit{typical} for the Bernstein component $\RRs$, if for every 
irreducible object $\pi$ in $\RR$, $\Hom_K(\tau, \pi)\neq 0$ implies that $\pi$ lies in $\RRs$. We say that $\tau$ is a \textit{type} for $\RRs$ if it is typical and 
$\Hom_K(\tau, \pi)\neq 0$ for every irreducible object $\pi$ in $\RRs$. Given $\RRs$, there exists a type $\tau$, unique up to isomorphism, except when 
$\RRs$ contains $\chi\circ \det$. In this case, there are two typical representations $\theta\circ \det$ and $\St\otimes \theta\circ \det$, where
$\theta:=\chi|_{\oF^{\times}}$ and $\St$ is the lift to $K$ of the Steinberg representations of $\GL_2(\Fq)$, see \cite{henniart}. For us a $\Qp$-rational representation of 
$G$, is a representation $W$ of the form
\begin{equation}\label{qprational}
\bigotimes_{\sigma: F\hookrightarrow L} (\Sym^{r_{\sigma}} L^2 \otimes \dt^{a_{\sigma}})^{\sigma},
\end{equation}
where $r_{\sigma}$, $a_{\sigma}$ are integers, $r_{\sigma}\ge 0$, and an element $\bigl(\begin{smallmatrix} a & b\\ c & d\end{smallmatrix}\bigr )$ in $G$ acts on the 
$\sigma$-component  via 
$\bigl(\begin{smallmatrix} \sigma(a) & \sigma(b)\\ \sigma(c) & \sigma(d)\end{smallmatrix}\bigr)$, see \cite[\S2]{bsch} for a proper setting. The locally $\Qp$-rational 
representations in the title refer 
to the representations of the form $\pi\otimes_L W$, where $\pi$ is a smooth representation of $G$ on an $L$-vector space and $W$ is a 
$\Qp$-rational representation as above.

\begin{thm}\label{A} Assume $p>2$. Let $\tau$ be a smooth absolutely irreducible $L$-representation of $K$, 
which is typical for the Bernstein component  $\RRs$. Let  $W$ be 
$\Qp$-rational representation of $G$, twisted by a continuous character. Let $M$ be a 
$K$-invariant lattice in $\tau\otimes W$. 
Let $\kappa$ be an absolutely irreducible smooth admissible $k$-representation of $G$, 
such that $\pif$ acts trivially. Suppose that there exists an  irreducible smooth 
$k$-representation $\sigma$ of $K$, such that 
\begin{itemize}
\item[(1)] $\Hom_K(\sigma, \kappa)\neq 0$;
\item[(2)] $\sigma$ occurs as a subquotient of $M\otimes_A k$.
\end{itemize}
Then there exists a finite extension $L'$ of $L$, an absolutely irreducible smooth $L'$-representation $\pi$ of $G$ in $\RRs$, and an admissible unitary 
$L'$-Banach space representation $(E,\|\centerdot\|)$ of $G$, such that the following hold:
\begin{itemize}
\item[(i)] $\pi\otimes_{L'} W_{L'}$ is a dense $G$-invariant subspace of $E$;
\item[(ii)] $\Hom_G(\kappa\otimes_k k', E^0\otimes_{A'} k')\neq 0$,
\end{itemize} 
where $E^0$ is the unit ball in $E$ with respect to $\|\centerdot\|$.
\end{thm}
We also have a variant of Theorem \ref{A}, when $\tau$ is the trivial representation of $K$, which allows $\pi$ to be possibly reducible unramified principal 
series representation, see Corollary \ref{cor2}. We do not know in general whether these completions are of finite length, and we can not control $\pi$, except that 
we know that $\pi$  lies in $\RRs$.  However, we show that 
any admissible unitary completion arises from our construction, see Proposition \ref{converse} and Lemma \ref{trivialll}. 
We also show that given $\kappa$ as above, there 
exists a unitary admissible topologically irreducible $L$-Banach space representation $E$ of $G$, such that $\Hom_G(\kappa, E^0\otimes_A k)\neq 0$, where 
$E^0$ is a unit ball in $E$ with respect to a $G$-invariant norm defining the topology on $E$, see Corollary \ref{cor3}. This result means that if one 
decides to throw away some irreducible smooth $k_L$-representations of $G$ by declaring them 'non-arithmetic', one is also forced to throw away 
some irreducible admissible unitary $L$-Banach space representations of $G$. When $\RRs$ contains a principal series representation, we show that 
in most cases the completions we get are not 'ordinary', for example when $\kappa$ is supersingular. 
Topologically irreducible completions of locally $\Qp$-rational representations are expected to be related to the $2$-dimensional representations of the absolute 
Galois group of $F$, see \cite{bsch}. If $F=\Qp$ this is indeed the case, see for example \cite{bergerbreuil}, \cite{colmez04b}, \cite{montreal}. If $F\neq \Qp$
then there is not so much known about the completions of locally $\Qp$-rational representations, with the exception of Vign\'eras paper \cite{vigint}. However, 
the $G$-invariant lattices in $\pi\otimes W$, that one gets in \cite{vigint} are always finitely generated over $A[G]$, it is expected that the completion 
with respect to such lattices will not be admissible in general. 
 
If $F=\Qp$ and $\RRs$ contains a principal series representation then the results of Berger-Breuil \cite{bergerbreuil} imply that the completions 
we get are topologically irreducible. Moreover, using results of Berger, Breuil and Colmez we may then transfer the statement of Theorem \ref{A} to the
Galois side.  We will describe this in more detail. Recall that a representation $V$ of $\mathcal G_{\Qp}:=\Gal(\Qbar/\Qp)$ is crystabelline if it becomes 
crystalline after restriction to $\Gal(\Qbar/E)$, where $E$ 
is an abelian extension of $\Qp$. Absolutely irreducible $L$-linear $2$-dimensional crystabelline representations of $\mathcal G_{\Qp}$  
with Hodge-Tate weights $(0,k-1)$, ($k\ge 2$) can be parameterized by pairs of smooth characters $\alpha, \beta:\Qp^{\times} \rightarrow L^{\times}$, such that 
$-(k-1)< \val(\alpha(p))\le \val(\beta(p))<0$ and $\val(\alpha(p))+\val(\beta(p))=-(k-1)$, see \cite[Prop 2.4.5]{bergerbreuil} or 
\cite[\S 5.5]{colmez04b}. We denote by $V(\alpha, \beta)$ the unique crystabelline representation $V$, such that 
$D_{cris}(V)=D(\alpha, \beta)$, where $D(\alpha, \beta)$ is the  filtered 
admissible $L$-linear $(\varphi, \mathcal G_{\Qp})$-module defined in \cite[Def 2.4.4]{bergerbreuil}. We denote by $\overline{V}$ the semi-simplification 
of the reduction modulo $\MM$ of any $\GG_{\Qp}$-stable lattice in $V$, let $\II_{\Qp}$ be the inertia subgroup of $\GG_{\Qp}$ and $\omega$ the reduction 
modulo $p$ of the cyclotomic character.  

\begin{thm}\label{B} Assume $p>2$. Fix smooth characters $\theta_1, \theta_2:\Zp^{\times}\rightarrow L^{\times}$, and an integer $k\ge 2$, such that 
\begin{itemize} 
\item[(a)] if $\theta_1=\theta_2$ then assume $k\ge p^2+1$;
\item[(b)] if $\theta_1\neq \theta_2$ and $\theta_1\theta_2^{-1}$ is trivial on $1+p\Zp$ then assume $k\ge p$.
\end{itemize}
Let $\rho$ be a semisimple  $2$-dimensional $k_L$-representation of $\mathcal G_{\Qp}$, such that 
\begin{itemize} 
\item[(c)] $\det \rho|_{\mathcal I_{\Qp}}= \bar{\theta_1} \bar{\theta_2}\omega^{k-1}$;
\item[(d)] if $\rho$ is irreducible, then it is absolutely irreducible;
\item[(e)] $\rho|_{\II_{\Qp}}\not\cong \bar{\theta}_1\oplus \bar{\theta}_2 \omega^{k-1}$ and 
$\rho|_{\II_{\Qp}}\not \cong \bar{\theta}_2\oplus \bar{\theta}_1 \omega^{k-1}$.
\end{itemize}
Then there exists a finite extension $L'$ of $L$ and an absolutely 
irreducible $2$-dimensional crystabelline $L'$-representation $V:=V(\alpha, \beta)$ of $\mathcal G_{\Qp}$, such that
\begin{itemize} 
\item[(i)] $\overline{V}\cong \rho$;
\item[(ii)] Hodge-Tate weights of $V$ are  $(0,k-1)$;
\item[(iii)]  either ($\alpha|_{\Zp^{\times}}=\theta_1$ and $\beta|_{\Zp^{\times}}=\theta_2$) or ($\alpha|_{\Zp^{\times}}=\theta_2$ and $\beta|_{\Zp^{\times}}=\theta_1$).
\end{itemize}
\end{thm}
See Theorem \ref{mainQp} for all $k\ge 2$. In the Example in \S\ref{FequalsQp} we check that our theory matches the known reductions of crystalline representations 
of small weights. To get the result when  $\theta_1=\theta_2$ we need to get around the case of equal Frobenius eigenvalues, which is not treated in the literature,
 this is done in \cite{except}. We comment on the assumptions on $\rho$ in Theorem \ref{B}: (c) is necessary, (d) 
is not serious, since if $\rho$ is irreducible then it is either absolutely irreducible or becomes reducible semi-simple after replacing $k_L$ with a finite extension, 
we impose (e) to make sure that we stay out of the 'ordinary' case, i.e. to ensure that the representation $V$ we get is absolutely irreducible. Since we cannot control 
$\pi$ in Theorem \ref{A}, we cannot control $\alpha(p)$ and $\beta(p)$ in Theorem \ref{B}. The condition $\pi$ lies in $\RRs$ in Theorem \ref{A} translates into 
condition (iii) in Theorem \ref{B}.

We will sketch the proof of Theorem \ref{A}. Let $e$ be an edge on the Bruhat-Tits tree, containing a vertex $v$. Let $\KK_1$ be the $G$-stabilizer of
$e$ and $\KK_0$ the $G$-stabilizer of $v$. The key point in our construction is that $G$ is an amalgam of $\KK_0$ and $\KK_1$ along $\KK_0\cap \KK_1$, which is the
stabilizer of $e$ preserving the orientation. This is used in \cite{coeff} and \cite{bp} to construct irreducible $k_L$-representations. We may assume that 
$\KK_0=KZ$, where $Z\cong F^{\times}$ is the centre of $G$. Let $\kappa$ be as in Theorem \ref{A}, then in \cite{bp} it is shown that there exists a $G$-equivariant 
injection $\kappa\hookrightarrow \Omega$, such that $\pif\in Z$ acts trivially on $\Omega$, $\Omega|_K$ is an injective envelope of $\kappa$ in the category $\Rep_{k_L} K$ 
of smooth $k_L$-representation of $K$.(For injective/projective envelopes see \S\ref{injectiveprojective}.) Since $\kappa$ is admissible, so is $\Omega$, 
moreover $\soc_G \Omega\cong \kappa$. Recall that the socle $\soc$ is the maximal 
semi-simple subobject. The first step is to lift $\Omega$ to a unitary admissible Banach space representation, \S\ref{sectionliftomega}.

We start with the general discussion, the details are contained in \S\ref{modules}, \S\ref{banach}. Let $\GG$ be a compact $p$-adic analytic group and 
let $I$ be an admissible injective object in $\Rep_{k_L} \GG$. 
Then dually $I^{\vee}$ is a projective finitely generated module of the completed group algebra $k[[\GG]]$, we may lift $I^{\vee}$ to a projective finitely 
generated module $P$ of $A[[\GG]]$. This module $P$ is then unique up to isomorphism. To $P$ following Schneider-Teitelbaum \cite{iw} we may associate 
a unitary admissible $L$-Banach space representation $P^d$ of $\GG$, $P^d:=\Hom_A^{cont}(P, L)$, with the supremum norm. If we let $(P^d)^0$ be the unit ball 
in $P^d$, then $(P^d)^0\otimes_A k_L\cong I$ as  $G$-representations. Concretely, when $\GG$ is a pro-$p$ group, then the only 
irreducible smooth $k_L$-representation
is the trivial one, and so $I$ is a finite direct sum of injective envelopes of the trivial representation. If $I$ is an injective envelope of the trivial 
representation then $I\cong C(\GG, k_L)$, the space of continuous function from $\GG$ to $k_L$, $I^{\vee}\cong k_L[[\GG]]$, $P\cong A[[\GG]]$ and 
$P^d\cong C(\GG, L)$, the space of continuous functions from $\GG$ to $L$ with the supremum norm.

We now go back to $\Omega$. For $i\in\{0,1\}$, set $\GG_i:=\KK_i/\pif^{\ZZ}$, since $\pif$ acts trivially, $\Omega$  is a representation of $\GG_i$. Denote 
the restriction of $\Omega$ to $\GG_i$ by  $\Omega_i$. The assumption $p\neq 2$ implies that the pro-$p$ 
Sylow subgroup of $\GG_1$ is equal to 
the pro-$p$ Sylow subgroup of $\GG_1\cap\GG_0$, which is a pro-$p$ Sylow subgroup of $\GG_0$.
This implies that $\Omega_1$ is an admissible injective object in $\Rep_{k_L} \GG_1$. 
The argument above gives us a projective finitely generated module $P_i$ of 
$A[[\GG_i]]$. Using some general facts about projective modules we find a $(\GG_0\cap \GG_1)$-equivariant isomorphism $\phi: P_0\cong P_1$, such that 
$\phi$ reduces to the identity modulo $\MM$. The results of \cite{iw} enable us to go back and forth between finitely generated $A[[\GG_i]]$-modules, and 
admissible unitary $L$-Banach space representations. So dually we get unitary $L$-Banach space representation $P_0^d$ of $\GG_0$, $P_1^d$ of $\GG_1$, and 
a $(\GG_1\cap \GG_0)$-equivariant isometrical isomorphism $\phi^d: P_1^d \overset{\cong}\rightarrow P_0^d$. We let $\pif$ act trivially everywhere, then by the amalgamation 
argument this data glues to a unitary admissible $L$-Banach space representation $B$ of $G$. Moreover, $B^0\otimes_A k_L\cong \Omega$ as a $G$-representation.
We note that although $P_0$ and $P_1$ are canonical, there is no canonical way to choose the isomorphism $\phi$. In general, different choices of $\phi$ will 
lead to non-isomorphic Banach space representations $B$, and different $\pi$ in Theorem \ref{A}. 

The second step is to produce $\pi$, see \S\ref{admissiblecompletions}. The assumption $\Hom_K(\sigma, \kappa)\neq 0$ implies that an injective envelope
$I_{\sigma}$ of $\sigma$ is a direct summand of $\Omega|_K$. This implies that the projective envelope $P_{\sigma^*}$ of $\sigma^*$ is a direct summand 
of $P_0$. We show that the assumption (2) in Theorem \ref{A} implies that 
$$\Hom_{L[[\GG_0]]}(P_0\otimes_A L, (\tau\otimes W)^*)\neq 0,$$ 
where $L[[\GG_0]]:=L\otimes_A A[[\GG_0]]$.  
Dually this means $\Hom_K(\tau\otimes W, B)\neq 0$. Since $\pif$ acts trivially on $B$ and by a scalar on $W$, there exists a unique extension $\tilde{\tau}$
of $\tau$ to a representation of $\KK_0$, such that $\Hom_{\KK_0}(\tilde{\tau}\otimes W, B)\neq 0$. Frobenius reciprocity then 
gives
\begin{equation}\label{C}
\Hom_G(\cIndu{\KK_0}{G}{\tilde{\tau}\otimes W}, B)\cong \Hom_G((\cIndu{\KK_0}{G}{\tilde{\tau}})\otimes W, B)\neq 0.
\end{equation}   
Using admissibility of $B$, we show that $B$ contains a $G$-invariant subspace of the form $\pi'\otimes W$, where $\pi'$ is a quotient of
 $\cIndu{\KK_0}{G}{\tilde{\tau}}$ of finite length. Thus if we replace $L$ with a finite extension, we may find a $G$-invariant subspace 
in $B$ isomorphic to $\pi\otimes W$, where $\pi$ is an absolutely irreducible smooth representation of $G$. Since $\tau$ is typical for 
$\RRs$, $\pi_{\overline{L}}$ is an object of $\RRs$. Since we work with coefficient fields which are not algebraically closed we use the 
results of Vign\'eras \cite{vig2}. Take $E$ to be the closure of $\pi\otimes W$ in $B$, then since $B$ is admissible, so is $E$ and 
we have an injection $E^0\otimes_A k_L\hookrightarrow B^0\otimes_A k\cong \Omega$. Since by construction $\soc_G \Omega\cong \kappa$, this yields the result. 

In general, it is quite hard to compute inside $B$. However, it might be possible to understand the completions better if we restrict ourselves  to the case when 
$\tau$ is the trivial representation, so that $\pi$ is an unramified principal series, and the weights in \eqref{qprational} are small,  $0\le r_{\sigma}\le p-1$. 
Using our methods one could try and lift the representations constructed 
in \cite{bp} to Banach space representations, (at least those that conjecturally correspond to the irreducible mod $p$ representations of $\Gal(\overline{F}/F)$), 
see Remark \ref{smoothtoalgebraic}. We hope to return to these questions in the future work.

\textit{Acknowledgements.} This paper grew out of the collaboration with Christophe Breuil, \cite{bp}. I thank Laurent Berger, Christophe Breuil, Ga\"etan Chenevier, 
Pierre Colmez, Matthew Emerton, Guy Henniart and Peter Schneider for answering my questions. I thank Florian Herzig for 
sending me a detailed list of comments on the earlier draft, his suggestions led to improvement of the exposition, especially in \S\ref{banach}. 
The paper was written when I was visiting IH\'ES and 
Universit\'e Paris-Sud, supported by Deutsche Forschungsgemeinschaft. I would like to thank these institutions. 

\section{Notation} 
Let $F$ be a finite extension of $\Qp$ with the ring of integers $\oF$, maximal ideal $\pF$ and the residue field isomorphic to $\Fq$. We fix a 
uniformizer $\varpi$ of $F$. Let $G:=\GL_2(F)$, $B$ the subgroup of upper-triangular matrices, $U$ the subgroup of unipotent upper-triangular matrices,
$K:=\GL_2(\oF)$, 
$$I:=\begin{pmatrix} \oF^{\times} & \oF \\ \pF & \oF^{\times}\end{pmatrix}, \quad I_1:=\begin{pmatrix} 1+\pF & \oF \\ \pF & 1+\pF\end{pmatrix}, \quad
K_1:=\begin{pmatrix} 1+\pF & \pF \\ \pF & 1+\pF\end{pmatrix}.$$
Let 
$$s:=\begin{pmatrix} 0  & 1\\ 1 & 0 \end{pmatrix}, \quad \Pi:=\begin{pmatrix} 0  & 1\\ \pif & 0 \end{pmatrix}, \quad t:=\begin{pmatrix} \pif  & 0\\  0 & 1 \end{pmatrix}.$$
Let $Z$ be the centre of $G$, $Z\cong F^{\times}$. Let $\KK_0$ be the $G$-normalizer of $K$, so that $\KK_0=KZ$ and let $\KK_1$ be the $G$-normalizer of 
$I$, so that $\Pi$ and $I$ generate $\KK_1$ as a group.

We fix an algebraic closure $\Qpbar$ of $\Qp$. We let $\val$ be the valuation on $\Qpbar$ such that $\val(p)=1$, and we set $|x|:=p^{-\val(x)}$. Let 
$L$ be a finite extension of $\Qp$ contained in $\Qpbar$, $A$ the ring of integers of $L$, $\varpi_L$ a uniformizer, and $\MM$ the maximal ideal of $A$, $k=k_L$
the residue field. The field $L$ will be our coefficient field, when needed we replace $L$ by a finite extension. Let $\Sigma$ be the set of $\Qp$-linear 
embeddings $F\hookrightarrow L$, we assume $[F:\Qp]=|\Sigma|$. Note that $k_L$ contains $\mathbb F_{q}$. As a consequence every irreducible 
smooth $k_L$-representation of $K$ is absolutely irreducible.  
If $M$ is an $A$-module then we write $\overline{M}:=M\otimes_A k$. If $V$ is a vector space over some field $\mathrm F$ 
we write $V^*:=\Hom_{\mathrm F}(V, \mathrm F)$ and if $\mathrm F'$ is a field extension of $\mathrm F$, then 
$V_{\mathrm F'}:= V\otimes_{\mathrm F}\mathrm F'$.
 
\section{Injective and projective envelopes}\label{injectiveprojective}
We recall some standard facts about injective and projective envelopes. Let $\ab$ be an abelian 
category. A monomorphism $\iota: N\hookrightarrow M$ is \textit{essential} if for every non-zero subobject $M_1$ of $M$ we have $N\cap M_1\neq 0$.
An injective envelope of $M$ is an essential monomorphism $\iota:M\hookrightarrow I$, such that $I$ is an injective object in $\ab$. 
An epimorphism $q: M\twoheadrightarrow N$ is \textit{essential} if for every morphism $s: P\rightarrow M$, the assertion '$qs$ is an epimorphism'
implies that $s$ is an epimorphism. A projective envelope of $M$ is an essential epimorphism $q: P\twoheadrightarrow M$ with $P$ a projective
object in $\ab$. One may easily
verify that injective and projective envelopes (if they exist) are unique up to (non-unique) isomorphism. So by abuse of language we 
will forget the morphism and say $I$ is an injective envelope of $M$ or $P$ is a projective envelope of $M$. 

\begin{thm} Let $R$ be a ring and $\Mod(R)$ the category of left $R$-modules, then every object in $\Mod(R)$ has an injective 
envelope.
\end{thm}
\begin{proof} \cite{mac} Theorem 11.3.
\end{proof}

\begin{lem}\label{triv} Let $\ab'$ be a full abelian subcategory of $\ab$ and assume that we have 
a functor $F:\ab\rightarrow \ab'$, which is right adjoint to the inclusion $i: \ab'\rightarrow \ab$.
Let $M$ be an object  in $\ab'$ and suppose that $M$ has an injective envelope $i(M)\hookrightarrow I$
in $\ab$, then $M\hookrightarrow F(I)$ is an injective envelope of $M$ in $\ab'$.
\end{lem}
\begin{proof} We have $\Hom_{\ab'}(M, N)=\Hom_{\ab}(i(M), i(N))= \Hom_{\ab'}(M, F(i(N)))$. So $F\circ i$ is right adjoint to 
the identity, and hence $N$ is canonically isomorphic to $F(i(N))$ for all $N$ in $\ab'$. Let  
$\alpha: i(M)\rightarrow N$ be a morphism in $\ab$. Then for all $M'$ in $\ab'$ the map 
$\alpha^*:\Hom_{\ab}(i(M'), i(M))\rightarrow \Hom_{\ab}(i(M'), N)$ is an injection if and only if  
the map $F(\alpha)^*:\Hom_{\ab'}(M', M)\rightarrow \Hom_{\ab'}(M', F(N))$ is an injection. So $F$ maps monomorphism to monomorphism, 
and for all $N$ in $\ab$ the canonical map $\iota_F: i(F(N))\rightarrow N$ is a monomorphism. Moreover, given 
$\phi\in \Hom_{\ab}(i(M), N)$, we have $\phi= \iota_F\circ i(F(\phi))$, as $F(\iota_F\circ i(F(\phi)))=\id_{\ab'} \circ F(\phi)= F(\phi)$.
Let $i(M)\hookrightarrow N$ be an essential monomorphism in $\ab$, then it factors through $i(M)\hookrightarrow i(F(N))\hookrightarrow N$,
which implies that $M\hookrightarrow F(N)$ is an essential monomorphism in $\ab'$. Since $I$ is injective
the functor $\Hom_{\ab}(\centerdot, I)$ is exact, hence the functor $\Hom_{\ab'}(\centerdot, F(I))$ is exact so $F(I)$ is injective in $\ab'$.
\end{proof} 
  
Using the theorem and the lemma one can obtain a lot of injective envelopes. We give some examples.

1) Let $\GG$ be a topological group. We say that a representation of $\GG$ on a $k$-vector space $V$ is smooth (or discrete), if the action 
of $\GG$ on $V$ is continuous, for the discrete topology on $V$. This is equivalent to saying that for all $v\in V$ the stabilizer 
of $v$ is an open subgroup of $\GG$. We denote the category of smooth $k$-representations of $\GG$ by $\Rep_k(\GG)$. We may view $\Rep_k(\GG)$ 
as a full subcategory of $\Mod(k[\GG])$. If $M$ is in $\Mod(k[\GG])$ we let $F(M)$ be a submodule of $M$ consisting of $v\in M$ such that 
the orbit map $\GG\rightarrow M$, $g\mapsto gv$ is continuous, for the discrete topology on $M$. Then $F$ satisfies the conditions of 
the Lemma \ref{triv}, and hence every object in $\Rep_k(\GG)$ has an injective envelope.

2) Let  $\DD_A(\GG)$ be the category of $p$-torsion $A$-modules $M$ with a continuous action of $\GG$, for the discrete topology on $M$.
Then $\DD_A(\GG)$ is a full subcategory of $\Mod(A[\GG])$, and $F(M)$  is a submodule of $M$ consisting of $v$, 
for which the orbit map is continuous and which are killed by some power of $p$. Again we obtain that every object in $\DD_A(\GG)$ has
an injective envelope. 

Projective envelopes are harder to come by, but in our situation we have two abelian categories $\ab$ and $\ab^{\vee}$ and 
a functor $M\mapsto M^{\vee}$ which induces an anti-equivalence of categories between $\ab$ and $\ab^{\vee}$. One may check that 

\begin{lem} A monomorphism $M\hookrightarrow I$ in $\ab$ is an injective envelope of $M$ in $\ab$ if and only if 
$I^{\vee}\twoheadrightarrow M^{\vee}$ is a projective envelope of $M^{\vee}$ in $\ab^{\vee}$.
\end{lem}  

In the following let $\GG$ be a profinite group with an open pro-$p$ subgroup. We discuss the structure of injective envelopes in $\Rep_k(\GG)$.

\begin{defi} A smooth $k$-representation $V$ of $\GG$ is called admissible if for every  open pro-$p$ subgroup $\PP$ of $\GG$ the subspace 
$V^{\PP}:=\{v\in V: gv= v, \forall g\in \PP\}$ is finite dimensional.
\end{defi} 
In fact, it is enough to check this for one open pro-$p$ group, see \cite{coeff} Theorem 6.3.2. If $\PP$ is an open normal pro-$p$ group
of $\GG$ and $S$ is irreducible, then $\PP$ acts trivially on $S$, since $S^{\PP}$ is a non-zero subrepresentation of $S$. The irreducible representations
of $\GG$ coincide with the irreducible representations of a finite group $\GG/\PP$. In particular, the set $\Irr(\GG)$ of the irreducible representations is
finite. 

\begin{lem}\label{nice} Let $V$ be a smooth representation of $\GG$ then $V$ is admissible if and only if the space $\Hom_{\GG}(S,V)$ is finite dimensional 
for all irreducible representations $S$ of $\GG$.
\end{lem}
\begin{proof} Let $\PP$ be an open normal pro-$p$ subgroup of $\GG$, then since $\PP$ acts trivially on $S$, we have 
$\Hom_{\GG}(S, V)\cong \Hom_{\GG}(S, V^{\PP})$ and this space is finite dimensional. Suppose that $\Hom_{\GG}(S,V)$ is finite dimensional 
for all irreducible representations $S$. Then arguing inductively we get that $\Hom_{\GG}(M, V)$ is finite dimensional for all representations 
$M$ of finite length. In particular, 
$$V^{\PP}\cong \Hom_{\PP}(\Eins, V)\cong \Hom_{\GG}(\Indu{\PP}{\GG}{\Eins}, V)$$
is finite dimensional.
\end{proof}

\begin{lem}\label{decomposeinj} Let $V$ be an admissible smooth representation of $\GG$. For each irreducible representations $S$ of $\GG$ set 
$m_S:=\dim_k\Hom_{\GG}(S,V)$. Let $V\hookrightarrow I$ be an injective envelope of $V$ in $\Rep_k(\GG)$. Then $I$ is admissible and 
$I\cong \oplus_{S\in \Irr(\GG)} I_S^{\oplus m_S}$, where $I_S$ is an injective envelope of $S$ in $\Rep_k(\GG)$.
\end{lem} 
\begin{proof} For each irreducible $S$ we have $\Hom_{\GG}(S,V)\cong \Hom_{\GG}(S,I)$, otherwise $S$ would be a nonzero subspace of $I$, such 
that $S\cap V=0$. Lemma \ref{nice} implies that $I$ is admissible. So the maximal semisimple subobject $\soc_{\GG} I$ (socle) of $I$,  is isomorphic 
to $\oplus S^{\oplus m_S}$. Now $I$ is an essential extension of $\soc_\GG I$, since if $W$ is a $\GG$-invariant subspace of $I$ such that 
$W\cap \soc_{\GG} I=0$ then $\soc_{\GG} W=0$. This implies that if  $\PP$ is an open normal subgroup of $\GG$ then $W^{\PP}=0$, and hence $W=0$.
One easily checks that $\oplus_{S\in \Irr(\GG)} I_S^{\oplus m_S}$ is an injective envelope of $\oplus S^{\oplus m_S}$, and the uniqueness of 
injective envelopes implies the claim.
\end{proof}

\begin{lem}\label{injtrivial} Let $\PP$ be a pro-$p$ group and let $C(\PP,k)$ be the space of continuous functions form $\PP$ to $k$. Then 
$\Eins \hookrightarrow C(\PP,k)$ is an injective envelope of $\Eins$ in $\Rep_k(\GG)$.
\end{lem}
\begin{proof} Now $C(\PP,k)^{\PP}$ is just the space of constant functions, so it is one dimensional. Hence, $\Eins\hookrightarrow C(\PP,k)$
is an essential monomorphism. If $V$ in $\Rep_k(\GG)$ then $\Hom_{\GG}(V, C(\PP,k))\cong \Hom_k(V,k)$ by Frobenius reciprocity.
Hence  the functor $\Hom_{\GG}(\centerdot, C(\PP,k))$ is exact and so $C(\PP,k)$ is injective.
\end{proof}

\begin{lem}\label{rat0} Let $\GG$ be a profinite group with an open pro-$p$ subgroup.  Suppose that every 
irreducible smooth $k$-representation of $\GG$ is absolutely irreducible. Let $S$ be an irreducible smooth $k$-representation 
of $\GG$ and $I_S$ be an injective envelope of $S$ in $\Rep_k \GG$. Let $k'$ be an extension of $k$, then $I_S\otimes_k k'$ is an injective 
envelope of $S\otimes_k k'$ in $\Rep_{k'} \GG$.
\end{lem}
\begin{proof} Lemma \ref{decomposeinj} gives an isomorphism $C(\GG, k)\cong \oplus_S I_S^{\dim_k S}$. Now 
$$C(\GG, k')\cong C(\GG, k)\otimes_k k'\cong \bigoplus_S (I_S \otimes_k k')^{\dim_k S}.$$
Since $I_S\otimes_k k'$ is a direct summand of an injective object in $\Rep_{k'} \GG$, it is also injective. Since 
we have an injection $S\otimes_k k'\hookrightarrow I_{S}\otimes_k k'$, the injective envelope $I_{S\otimes_k k'}$ of $S\otimes_k k'$ will be 
isomorphic to a direct summand of $I_{S}\otimes_k k'$. By assumption $S\otimes_k k'$ is irreducible, hence 
 $C(\GG, k')\cong \oplus_S I_{S\otimes_k k'}^{\dim_k S}$. This implies $I_{S\otimes_k k'}\cong I_S\otimes_k k'$.
\end{proof} 
 
\begin{lem}\label{ratsoc} Let $\GG$ be a profinite group with an open pro-$p$ subgroup. Let $V$ be an admissible $k$-representation of $\GG$, 
and let $k'$ be an extension of $k$. Suppose that every irreducible $k$-representation of $\GG$ is absolutely irreducible then 
$(\soc_{\GG} V)\otimes_k k'\cong \soc_{\GG}(V\otimes_k k')$.
\end{lem}
\begin{proof} For an irreducible smooth $k$-representation $S$ of $\GG$, set 
$$m_S:=\dim_k \Hom_{\GG}(S, V), \quad m'_S:=\dim_{k'}\Hom_{\GG}(S\otimes_k k', V\otimes_k k').$$
 We have to show that $m_S=m'_S$ for all $S$. Clearly $m_S\le m'_S$. 
Let $I$ be an injective envelope of $V$ in $\Rep_k \GG$ and let $I'$ be an  injective envelope of $V\otimes_k k'$ in 
$\Rep_{k'} \GG$. It follows from Lemmas \ref{decomposeinj} and \ref{rat0} that $I\otimes_k k'$ is injective in $\Rep_{k'} \GG$.
Since we have an injection $V\otimes_k k'\hookrightarrow I\otimes_k k'$, $I'$ is isomorphic to a direct summand of $I\otimes_k k'$. 
Lemmas \ref{decomposeinj} and \ref{rat0} imply that $m'_S\le m_S$. 
\end{proof} 

 \begin{lem}\label{rat1} Let $\PP$ be a pro-$p$ group and  let $V$ be a smooth admissible $k$-representation of $\PP$. 
Let $k'$ be an extension of $k$, then 
$V^{\PP}\otimes_k k'\cong (V\otimes_k k')^{\PP}$.
\end{lem}       
\begin{proof} The only irreducible smooth $k$-representation of $\PP$ is the trivial representation, which is absolutely irreducible.
Since $\soc_{\PP}V=V^{\PP}$ Lemma \ref{ratsoc} implies the assertion.
\end{proof}

\section{Modules over completed group algebras}\label{modules}
Let $\GG$ be a pro-finite group with an open pro-$p$ subgroup. We define the completed group algebras:
$$A[[\GG]]:=\underset{\leftarrow}\lim \ A[\GG/\PP]\cong \underset{\leftarrow}{\lim} \  A/\MM^n[\GG/\PP], \quad  k[[\GG]]:=\underset{\leftarrow}\lim \ k[\GG/\PP],$$
where the limit runs over all open normal pro-$p$ subgroups and natural numbers $n$. We put the discrete topology on $A/\MM^n[\GG/\PP]$ (resp. $k[\GG/\PP]$) and inverse 
limit topology on $A[[\GG]]$ (resp. $k[[\GG]]$). So for all open normal pro-$p$ subgroups $\PP$ and all $n\ge 1$ the kernels
 of $A[[\GG]]\rightarrow A/\MM^n[\GG/\PP]$
(resp. $k[[\GG]]\rightarrow k[\GG/\PP]$) form a basis of open neighbourhoods of zero. Since $k$ is a finite field $A/\MM^n[\GG/\PP]$ (resp. $k[\GG/\PP]$) is 
finite, hence $A[[\GG]]$ (resp. $k[[\GG]]$) is compact. In the following let $\Lambda$ be either $A[[\GG]]$ or $k[[\GG]]$.  
Let $\CC(\Lambda)$
denote the category of topological, Hausdorff, complete $\Lambda$-modules $M$, such that $M$ has a system of open neighbourhoods
of $0$ consisting of submodules $N$, for which $M/N$ has finite length, with 
morphisms continuous $\Lambda$-homomorphisms. Since $k$ is a finite field, such modules $M$ are compact. Recall that a topological 
$\Lambda$-module $M$ is linearly topological if $0$ has a fundamental system of open neighbourhoods consisting 
of $\Lambda$-submodules. Equivalently the category $\CC(\Lambda)$ could be defined as a category of linearly topological compact Hausdorff
$\Lambda$-modules. Let $\DD(\Lambda)$ be the category 
of discrete $\Lambda$-modules, a discrete module is always $p$-torsion. Moreover, $\DD(k[[\GG]])=\Rep_k(\GG)$.  

We recall some facts about these categories. The category $\CC(\Lambda)$ is abelian 
with exact inverse limits, \cite[Thm 3]{gab}. The category $\DD(\Lambda)$ is abelian with exact direct limits, \cite{bru} Lem 1.8. 
If  $M$ is an object of $\DD(\Lambda)$ or $\CC(\Lambda)$, then define $M^{\vee}:=\Hom_{A}^{cont}(M, L/A)$, with the 
discrete topology on $L/A$ and compact open topology on $M^{\vee}$. Then $M^{\vee\vee}\cong M$ and  
the functor $\Hom_{A}^{cont}(\centerdot, L/A)$ induces an anti-equivalence of categories between $\CC(\Lambda)$ and $\DD(\Lambda)$. 

Every object in $\DD(\Lambda)$ has an injective envelope, dually every object in 
$\CC(\Lambda)$ has a projective envelope. Let $M$ and $N$ be projective 
objects in $\CC(\Lambda)$. Denote by $\rad(M)$ the intersection of all maximal subobjects in $M$. Then 
$M\cong N$ if and only if $M/\rad(M)\cong N/\rad(N)$. Moreover, a projective object is indecomposable if and only if $M/\rad(M)$ is 
simple, \cite{demgab} V.2.4.6 b). This is equivalent to: a projective module is indecomposable if and only if it is a projective envelope
of an irreducible module.   Every indecomposable projective 
module in $\CC(\Lambda)$ is isomorphic to a direct summand of $\Lambda$, \cite{gab} \S3 Cor 1.

\begin{thm}[\cite{demgab} V.2.4.5]\label{proj}The following hold:
\begin{itemize} 
\item[(a)] every projective object in $\CC(\Lambda)$ is isomorphic to a direct product $\prod_{i\in I} P_i$, for 
some set $I$, with $P_i$ projective indecomposable objects;
\item[(b)] with the notations of (a), if $Q$ is  a projective object in $\CC(\Lambda)$ and $q:\prod_{i\in I} P_i\twoheadrightarrow Q$ is an
epimorphism then there exists a subset $J$ of $I$ such that $q_J: \prod_{i\in I} P_i\rightarrow Q\oplus (\prod_{j\in J} P_j)$
induced by $q$ and the canonical projections $\prod_{i\in I} P_i\rightarrow P_j$ is an isomorphism;
\item[(c)] suppose that $\prod_{i\in I} P_i\cong \prod_{j\in J} Q_j$ with $Q_j$ projective indecomposable then 
there exists a bijection $h: I\rightarrow J$ such that $P_i\cong Q_{h(i)}$ for all $i$.
\end{itemize}
\end{thm}      

\begin{prop}\label{indecomp} Let $\PP$ be an open normal pro-$p$ subgroup of $\GG$, then the irreducible $\Lambda$-modules, coincide
with the irreducible $k[\GG/\PP]$-modules. In particular, there are only finitely many irreducibles, and as $A$-modules they are finite dimensional 
$k$-vector spaces. Moreover, there
exists an isomorphism of $\Lambda$-modules:
$$\Lambda\cong \bigoplus_{S} (\dim_k S) P_S,$$
where the direct sum is taken over all irreducible $\Lambda$-modules and $P_S$ is a projective envelope of $S$ in $\CC(\Lambda)$.
\end{prop} 
\begin{proof} Let $S$ be an irreducible $\Lambda$-module in $\CC(\Lambda)$. The anti-equivalence of 
categories implies that $S^{\vee}:=\Hom^{cont}_{A}(S, L/A)$ is an irreducible discrete $p$-torsion module of $\GG$.
Hence, the $A$-module structure of $S^{\vee}$ is just a $k$-vector space, and since $S^{\vee}$ is 
discrete and $\PP$ is pro-$p$, the subspace of $\PP$-invariants of $S^{\vee}$ is non-zero. Since $S^{\vee}$ is irreducible
we obtain that $\PP$ acts trivially on $S^{\vee}$. By dualizing back we get the result. Moreover, we have 
$$\Hom_{\Lambda}(\Lambda, S)\cong \Hom_{k[\GG/\PP]}(k[\GG/\PP], S)\cong S.$$
This implies that 
$$\Lambda/\rad(\Lambda)\cong  \bigoplus_{S} (\dim_k S) S,$$
and since a projective module in $\CC(\Lambda)$ is determined by its head, we get that 
$$\Lambda\cong \bigoplus_{S} (\dim_k S) P_S.$$
\end{proof}

From now on we assume that $\GG$ is a $p$-adic Lie group. Then it follows from results of Lazard \cite{laz} that 
 $\Lambda$ is noetherian, see \cite{otmar} Cor. 2.4. Hence, the category of finitely generated modules 
$\mfg{\Lambda}$ is abelian. Moreover, if $M$ is finitely generated over $\Lambda$ then there 
exists a unique Hausdorff topology on $M$ such that $M$ is a topological $\Lambda$-module, and every 
$\Lambda$-homomorphism between finitely generated modules is continuous for the canonical topology, \cite{iw} Prop 3.1 or 
\cite{nsw} \S2 Prop 5.2.22. Hence, we may view $\mfg{\Lambda}$ as a full subcategory of $\CC(\Lambda)$.

\begin{prop}\label{good} Every object in $\mfg{\Lambda}$ has a projective envelope. Moreover, Theorem \ref{proj}
holds for $\mfg{\Lambda}$ if we replace products with finite direct sums. 
\end{prop}
\begin{proof} A projective indecomposable object in $\CC(\Lambda)$ is a direct summand of $\Lambda$, 
hence lies in $\mfg{\Lambda}$. Let $M$ be a finitely generated $\Lambda$-module.  Then there exists a surjection  
$\alpha: \Lambda^n \twoheadrightarrow M$, for some integer $n$. Let $\beta:P\twoheadrightarrow M$ be  
a projective envelope of $M$ in $\CC(\Lambda)$. Since $\Lambda^n$ is projective there exists 
$\gamma: \Lambda^n\rightarrow P$ such that $\alpha=\beta\circ \gamma$. Since $\beta$ is essential, $\gamma$ is surjective,
and since $P$ is projective, $\gamma$ has a section. So $P$ is isomorphic to a direct summand of $\Lambda^n$. 
Proposition \ref{indecomp} and Theorem \ref{proj} implies that $P$ is isomorphic 
to a finite direct sum of projective indecomposable modules. Hence $P$ is finitely generated. The same argument with 
$P=M$ gives that every finitely generated projective module in $\CC(\Lambda)$ is isomorphic to a finite direct sum 
of indecomposable projective modules. The last assertion follows from Theorem \ref{proj}. 
\end{proof} 

\begin{lem}\label{redproj} Let $S$ be an irreducible $A[[\GG]]$-module and let $P_S$ be a projective envelope of $S$ in $\mfg{A[[\GG]]}$ then 
$P_S\otimes_A k$ is a projective envelope of $S$ in $\mfg{k[[\GG]]}$.
\end{lem}
\begin{proof} Since as $A$-module $S$ is a $k$-vector space the map $P_S\rightarrow S$ factors through $P_S\otimes_A k$. This implies that 
$P_S\otimes_A k\rightarrow S$ is an essential epimorphism. Since $P_S$ is isomorphic to a direct summand of $A[[\GG]]$, $P_S \otimes_A k$ is 
isomorphic to a direct summand of $k[[\GG]]$, hence it is projective. 
\end{proof}

\begin{prop}\label{2surj} Let $P$ and $M$ be finitely generated $A[[\GG]]$-module. Suppose that $P$ is projective and we have two
surjective homomorphisms of $A[[\GG]]$-modules $\psi_1: P\rightarrow M$, $\psi_2: P\rightarrow M$. Then there exists $\phi\in \Aut_{A[[\GG]]}(P)$ such that 
$\psi_2\circ \phi=\psi_1$.
\end{prop}
\begin{proof} Let $P_M$ be a projective envelope of $M$ in $\mfg{A[[\GG]]}$. Since $\psi_i$ are surjective, for $i\in\{0,1\}$ there exists idempotents
$e_i\in \End_{A[[\GG]]}(P)$ such that $(1-e_i)P$ lies in the kernel of $\psi_i$, $e_i P\cong P_M$ and $\psi_i: e_i P\rightarrow M$ is a projective envelope of $M$. 
Since projective envelopes are unique up to isomorphism, there exists an isomorphism of $\phi_1: e_1P\cong e_2 P$, such that $\psi_2\circ \phi_1=\psi_1$.
It follow from Theorem \ref{proj} (c) that there exists an isomorphism of $A[[\GG]]$-modules $\phi_2: (1-e_1)P\cong (1-e_2)P$. Let 
$\phi:=(\phi_1, \phi_2)$ be a homomorphism $P=e_1P\oplus (1-e_1)P\rightarrow e_2P\oplus (1-e_2)P=P$. Then $\phi$ is an isomorphism and $\psi_2\circ \phi=\psi_1$. 
\end{proof}

\begin{prop}\label{Misproj} Let $M$ be a finitely generated $A[[\GG]]$-module, which is $A$-torsion free. 
Assume that $M\otimes_A k$ is a projective object
in $\CC(k[[\GG]])$. Then 
$$M\otimes_A k\cong \bigoplus_{S} m_S (P_S\otimes_A k), \quad M\cong \bigoplus_{S} m_S P_S,$$
 where the sum is taken 
over irreducible modules $S$, $m_S$ denotes some finite multiplicities, and $P_S$ denotes a projective envelope of $S$ in $\CC(A[[\GG]])$.
\end{prop}
\begin{proof} Since $M$ is finitely generated over $A[[\GG]]$, $M\otimes_A k$ is finitely generated over  $k[[\GG]]$. Proposition \ref{good}
implies that there exists uniquely determined non-negative integers $m_S$, such that  $M\otimes_A k\cong \oplus_{S} m_S \overline{P}_S$, 
where $\overline{P}_S$ is a projective envelope of $S$ in $\CC(k[[\GG]])$. Lemma \ref{redproj} implies that there exists an isomorphism
$P_S \otimes_A k\cong \overline{P}_S$. Set $P:=\oplus_{S} m_S P_S$. Since $P$ is projective there exists $\psi: P\rightarrow M$ making the 
following diagram commute:
\begin{equation}\label{diag}
 \xymatrix{ P\ar[r]^{\psi}\ar[d]& M\ar[d] \\P\otimes_A k \ar[r]^{\cong}& M\otimes_A k}
\end{equation}
Let $Q$ be the cokernel of $\psi$. Then $Q\otimes_A k=0$, and since $M$ is a compact $A$-module, $Q$ is a compact $A$-module.
 Nakayama's lemma \cite{SGA3} Exp. $VII_B$ (0.3.3) implies that $Q=0$. Hence $\psi$ is surjective. Since $M$ is $A$-torsion free, it
is a flat $A$-module. This implies that $(\Ker \psi)\otimes_A k =0$, which again by Nakayama's lemma gives $\Ker\psi =0$.   
\end{proof}
 
\begin{cor}\label{liftisomorphism} Let $P$ be a finitely generated projective $A[[\GG]]$-module then the reduction map 
$\Aut_{A[[\GG]]}(P)\rightarrow \Aut_{k[[\GG]]}(P\otimes_A k)$ is surjective.
\end{cor}
\begin{proof} Set $M=P$ in the diagram \eqref{diag}.
\end{proof}

\begin{prop}\label{multhom} Let $S$ be an irreducible 
$A[[\GG]]$-module and let $P_S$ be a projective envelope of $S$, let $M$ be an $A[[\GG]]$-module, 
such that $M$ as an $A$-module is free of finite rank. Then $\Hom_{A[[\GG]]}(P_S, M)$ is a free $A$-module of rank $m$, where $m$ is 
the multiplicity with which $S$ occurs as a subquotient of $\overline{M}$.
\end{prop}
\begin{proof} We claim that $\Hom_{A[[\GG]]}(P_S, \overline{M})$ is a 
$k$-vector space of dimension $m$. Since $M$ is a free $A$-module of finite rank, 
$\overline{M}$ is a finite dimensional $k$-vector space. In particular $\overline{M}$ is an $A[[\GG]]$-module of finite length. Suppose that 
$\overline{M}$ is irreducible, then since $P_S$ is a projective envelope of $S$, we have $\dim_k \Hom_{A[[\GG]]}(P_S, \overline{M})=1$ if 
$S\cong \overline{M}$, and $\dim_k \Hom_{A[[\GG]]}(P_S, \overline{M})=0$ if $S\not\cong \overline{M}$. In general, let $S'$ be an irreducible 
submodule of $\overline{M}$. Since $P_S$ is projective we get an exact sequence:
$$0\rightarrow \Hom_{A[[\GG]]}(P_S, S')\rightarrow \Hom_{A[[\GG]]}(P_S, \overline{M})\rightarrow  \Hom_{A[[\GG]]}(P_S, \overline{M}/S')\rightarrow 0.$$
We get the assertion by induction on the length of $\overline{M}$. Now $\Hom_{A[[\GG]]}(P_S, M)$ is a direct summand of 
$\Hom_{A[[\GG]]}(A[[\GG]], M)\cong M$. Hence,  $\Hom_{A[[\GG]]}(P_S, M)$ is a free $A$-module of finite rank. We have 
$$\Hom_{A[[\GG]]}(P_S, M)\otimes_A k \cong \Hom_{A[[\GG]]}(P_S, M\otimes_A k)\cong k^m.$$
Hence, $\Hom_{A[[\GG]]}(P_S, M)$ is a free $A$-module of rank $m$.
\end{proof}   
 
We set $L[[\GG]]:=A[[\GG]]\otimes_A L$, as $A[[\GG]]$ is noetherian, so is $L[[\GG]]$. Hence the category $\mfg{L[[\GG]]}$ of finitely generated 
$L[[\GG]]$-modules is abelian.

\begin{cor}\label{dim} Let $S$ be an irreducible $A[[\GG]]$-module, and let $P_S$ be a projective envelope of $S$ in $\mfg{A[[\GG]]}$. Let $V$ be an 
$L[[\GG]]$-module, such that $V$ is a finite dimensional $L$-vector space. Let $M$ be any $\GG$-invariant $A$-lattice in $V$ then 
$\dim_L\Hom_{L[[\GG]]}(P_S\otimes_A L, V)=m$, where $m$ is the multiplicity with which $S$ occurs in $\overline{M}$.
\end{cor}
\begin{proof} It follows from the discussion in \cite{iw} before Proposition 3.1, that 
  $$\Hom_{L[[\GG]]}(P_S\otimes_A L, V)\cong \Hom_{A[[\GG]]}(P_S, M)\otimes_A L.$$
The assertion follows from \ref{multhom}.
\end{proof}

\section{Banach space representations}\label{banach}
 Let $\GG$ be a compact $p$-adic Lie group. We recall some facts about Banach space representations of $\GG$. 
We follow closely Schneider-Teitelbaum \cite{iw}. Let $\Ban_L$ denote the category of $L$-Banach spaces. We note that we do not  
fix a norm defining the topology on the Banach space $E$, when we do want to fix such a norm 
$\|\centerdot\|$ we will write $(E, \|\centerdot\|)$. 

\begin{defi} An $L$-Banach space representation $E$ of $\GG$ is an $L$-Banach space $E$ together with a $\GG$-action 
by continuous linear automorphisms such that the map $\GG\times E\rightarrow E$ describing the action is continuous.
\end{defi}

Let $\Ban_L(\GG)$ be the category of $L$-Banach space representations with morphisms being all $\GG$-equivariant continuous 
linear maps. 

\begin{defi}\label{admissiblebanach} An $L$-Banach space representation $E$ of $\GG$ is called admissible if there exists a $\GG$-invariant 
bounded open $A$-submodule $M\subseteq E$ such that for any open pro-$p$ group $\PP$ of $\GG$, the $A$-submodule 
of $(E/M)^{\PP}$ is of cofinite type (i.e. $\Hom_A((E/M)^{\PP}, L/A)$ is a finitely generated $A$-module).
\end{defi} 
Let $\Ban_{L}^{adm}(\GG)$ be the full subcategory of $\Ban_{L}(\GG)$ consisting of admissible $L$-Banach representations of
$\GG$. It follows from \cite{iw} Theorem 3.5, that $\Ban_{L}^{adm}(\GG)$ is an abelian category. 

Recall that a topological $A$-module $M$ is linearly topological if $0$ has a fundamental system of open neighbourhoods consisting 
of $A$-submodules. Let $\Mod_{top}(A)$ be the category of all Hausdorff linearly topological $A$-modules, with morphisms all continuous 
$A$-linear maps. Let $\mfl(A)$ be the full subcategory of $\Mod_{top}(A)$ consisting of all torsion-free and compact linearly topological 
$A$-modules. The superscript $fl$ stands for flat. Following \cite{iw} we recall that an $A$-module is torsion-free if and only if it is flat,
\cite[I \S2.4 Prop 3(ii)]{bourbaki}; a compact linear-topological $A$-module $M$ is flat if and only if $M\cong \prod_{i\in I} A$ for some set $I$, 
\cite[$VII_B$ (0.3.8)]{SGA3}. Given $M$ in $\mfl(A)$ we define 
$$M^d:=\Hom^{cont}_{A}(M, L).$$
Now $M^d$ carries a structure of an $L$-Banach space with $\|\ell\|:=\max_{v\in M}|\ell(v)|$. Let $\mfgfl{A[[\GG]]}$ be the category of 
finitely generated $A[[\GG]]$-modules, which are $A$-torsion free.
Given $M$ in $\mfgfl{A[[\GG]]}$ we equip it with the canonical topology, then $M$ is an object in $\mfl(A)$. Given an additive category $\Aa$ we 
denote by $\Aa_{\QQ}$ the additive category with the same objects as $\Aa$ and $\Hom_{\Aa_{\QQ}}(A,B):=\Hom_\Aa(A,B)\otimes \QQ$.

\begin{thm}[\cite{iw}, Thm 1.2, Thm. 3.5]\label{duality} The functor $M\mapsto M^d$ induces an anti-equivalence of categories 
$$\mfl(A)_{\QQ}\overset{\sim}{\rightarrow} \Ban_L, \quad
\mfgfl{A[[\GG]]}_{\QQ}\overset{\sim}{\rightarrow} \Ban^{adm}_{\GG}(L).$$
\end{thm}   
Note that $\mfgfl{A[[\GG]]}_{\QQ}$ is equivalent to $\mfg{L[[\GG]]}$. Let $E$ be an $L$-Banach space the object $E^d$ in $\mfl(A)_{\QQ}$ 
corresponding to $E$ is constructed as follows, see the proof of \cite[Thm 1.2]{iw}. We may choose a norm $\|\centerdot\|$ defining the 
topology on $E$, and such that $\|E\|\subseteq |L|$. Let $E^0$ be the unit ball in $E$ with respect to $\|\centerdot\|$, set  
$$E^d:=\Hom_A(E^0, A),$$
with the topology of pointwise convergence, that is the coarsest locally convex topology such that for each $v\in E^0$ the map 
$E^d\rightarrow A$, $\phi\mapsto \phi(v)$ is continuous.

\begin{lem}\label{dualityvee} 
Let $M$ be in $\mfl(A)$ and  let $(M^d)^0$ be the unit ball in $M^d$ with respect to the supremum norm then there exists a canonical 
isomorphism $(M^d)^0\otimes_A k\cong (M\otimes_A k)^{\vee}$. Conversely let $(E, \|\centerdot\|)$ be an $L$-Banach space, assume that 
$\|E\|\subseteq |L|$. Let $E^0$ be the unit ball in $E$, and let 
$M:=\Hom_A(E^0, A)$ with the topology as above. Then there exists a canonical topological isomorphism $M\otimes_A k\cong (E^0\otimes_A k)^{\vee}$.
\end{lem}
\begin{proof} The reduction map $A\rightarrow k$ is continuous. Hence, we obtain a homomorphism of $A$-modules $r:\Hom^{cont}_A(M, A)\rightarrow 
\Hom^{cont}_A(M, k)$. We claim that $r$ is surjective. We note that the claim is clear if $M$ is of finite rank. In general $M\cong \prod_{i\in I} A$,
for some set $I$. So if $\phi\in \Hom^{cont}_A(M, k)$ then there exists a subset $J\subseteq I$ with $I\setminus J$ finite and an integer
 $n\ge 1$, 
such that $\prod_{j\in J} A \times \prod_{i\in I\setminus J} \MM^n$ is contained in the kernel of $\phi$, since such subsets form a basis of open 
neighbourhoods of $0$ in $\prod_{i\in I} A$. The problem reduces to showing that the map  $\Hom^{cont}_A(\prod_{i\in I\setminus J} A, A)\rightarrow 
\Hom^{cont}_A(\prod_{i\in I\setminus J} A, k)$ is surjective. Since $I\setminus J$ is finite we are done. The claim yields a short exact sequence 
of $A$-modules:
\begin{equation}\label{FH1}
0\rightarrow \Hom^{cont}_A(M, A)\overset{\varpi_L}{\rightarrow} \Hom^{cont}_A(M, A)\rightarrow  \Hom^{cont}_A(M, k)\rightarrow 0.
\end{equation}     
On the other hand, $\Hom^{cont}_A(M, A)$ is torsion-free and hence flat. So tensoring the short exact sequence 
$0 \rightarrow A\overset{\varpi_L}{\rightarrow} A \rightarrow k\rightarrow 0$ with $\Hom^{cont}_A(M, A)$ we obtain a short exact sequence:
\begin{equation}\label{FH2}
0\rightarrow \Hom^{cont}_A(M, A)\overset{\varpi_L}{\rightarrow} \Hom^{cont}_A(M, A)\rightarrow  \Hom^{cont}_A(M, A)\otimes_A k\rightarrow 0.
\end{equation}   
Now, \eqref{FH1} and \eqref{FH2} imply that the natural map 
$$\Hom^{cont}_A(M, A)\otimes_A k\rightarrow  \Hom^{cont}_A(M, k)$$ 
  is an isomorphism. Hence, $(M^d)^0\otimes_A k \cong  \Hom^{cont}_A(M\otimes_A k, k)\cong (M\otimes_A k)^{\vee}$; $(M^d)^0\otimes_A k$ carries the 
discrete topology, $M\otimes_A k$ is compact and so $(M\otimes_A k)^{\vee}$ also carries the discrete topology. 
The second part follows from Theorem \ref{duality}.
\end{proof}    

\begin{lem}\label{redimq} Let $(E,\|\centerdot\|)$ be an $L$-Banach space, such that $\|E \|\subseteq |L|$. 
Let $(E_1,\|\centerdot\|)$ be a closed subspace.
Then we have an exact sequence of $A$-modules:
\begin{equation}\label{FH3}
 0\rightarrow E_1^0\rightarrow E^0\rightarrow (E/E_1)^0\rightarrow 0,
\end{equation}
 \begin{equation}\label{FH4}
 0\rightarrow E_1^0\otimes_A k \rightarrow E^0\otimes_A k\rightarrow (E/E_1)^0\otimes_A k\rightarrow 0,
\end{equation}
where superscript $0$ denotes the unit ball in the respective Banach space.
\end{lem}
\begin{proof} The quotient space $E/E_1$ carries  a norm defined by 
$$\| v+ E_1\|:=\inf_{u\in E_1} \|v+u\|,\quad \forall v\in E.$$
 It is clear that $E^0$ maps into $(E/E_1)^0$. Since $L$ is discretely valued for every $v\in E$ there exists $u\in E_1$ such that $\| v+ E_1\|=\|v+u\|$. 
Hence, we obtain a surjection $E^0\twoheadrightarrow  (E/E_1)^0$. This implies \eqref{FH3}. Now $(E/E_1)^0$ is torsion-free and  hence flat. By tensoring
\eqref{FH3} with $\otimes_A k$ we obtain \eqref{FH4}.
\end{proof}

Let $V$ be an $L$-vector space and $M$ an $A$-submodule of $V$. We say that $M$ is a \textit{lattice} in $V$, if for every $v\in V$ there 
exists $x\in L^{\times}$ such that $xv\in M$, \cite{nfa} \S2. We say that $M$ is \textit{separated} if $\bigcap_{n\ge 0} \varpi_L^n M =0$.
If $E$ is an $L$-Banach space and $M$ is an open lattice in $E$ is separated if and only if it is bounded. Moreover, if $M$ is an 
open separated lattice in $E$ then the gauge of $M$, \cite{nfa} \S2, defined by
$$ \|v\|_M:= \inf_{v\in a M} |a|, \quad \forall v\in E$$ 
is a norm (since $M$ is separated), and the topology on $E$ defined by $\|\centerdot\|_M$ coincides with the original one (since $M$ is
open). If $E$ is an $L$-Banach space representation of $\GG$ then (since $\GG$ is compact) there exists an open separated $\GG$-invariant 
lattice $M$ in $E$, \cite{em1} Lemma 6.5.5, \cite{em2} Example 3.7, Lemma 3.9. Since $M$ is $\GG$-invariant we have $\|gv\|_M=\|v\|_M$ for all 
$v\in M$ and $g\in \GG$, so $E$ is a \textit{unitary} $L$-Banach space representation of $\GG$. 
Since $L$ is discretely valued $\|E\|_M\subseteq |L|$.

\begin{prop}\label{dproj} Let $(E, \|\centerdot\|)$ be a unitary  $L$-Banach space 
representation of $\GG$, such that $\|E\|\subseteq |L|$. Let $E^0$ be the unit
ball in $E$. Suppose that  $E^0\otimes_A k \cong I$, where $I$ is an injective 
admissible object in $\Rep_k(\GG)$. Let $m_S:=\dim_k\Hom_{\GG}(S, I)$ then  there exists a $\GG$-equivariant isometrical isomorphism 
$$(E, \|\centerdot\|)\overset{\cong}{\rightarrow} \bigoplus_{S\in \Irr(\GG)} (P_{S^*}^d)^{\oplus m_S},$$
where $P_{S^*}$ is the projective envelope of $S^*$ in $\CC(A[[\GG]])$, and the right hand-side is equipped with the supremum norm.

In particular, if $\PP$ is an open pro-$p$ subgroup of $\GG$ and $m:=\dim_k I^{\PP}$ then there exists a $\PP$-equivariant isometrical 
isomorphism
$$(E, \|\centerdot\|)\overset{\cong}{\rightarrow} C(\PP,L)^{\oplus m},$$
where $C(\PP,L)$ denotes the space of continuous functions from $\PP$ to $L$ with the supremum norm.
\end{prop}
\begin{proof} Since $E^0\otimes_A k$ is 
 admissible in $\Rep_k(\GG)$, it follows from \cite{em1} Proposition 6.5.7 that $E$ is an 
admissible $L$-Banach space representation of $\GG$ in the sense of Definition \ref{admissiblebanach}. 
Set $M:=\Hom_A(E^0, A)$, then 
Lemma \ref{dualityvee} implies that $M\otimes_A k\cong I^{\vee}$. It follows 
from Lemma \ref{decomposeinj} that $I\cong \oplus_S I_S^{\oplus m_S}$.
Since $I_S$ is injective $I^\vee_S$ is a projective $k[[\GG]]$-module. Since $I_S$ is admissible $I^{\vee}_S$ is finitely generated over $k[[\GG]]$, 
\cite{vignak} or the proof of \cite{em1} Proposition 6.5.7. Since $I_S$ is an injective envelope of $S$, $I_S^{\vee}$ is a projective envelope of 
$S^*$ in  $\mfg{k[[\GG]]}$. Proposition \ref{Misproj} implies that $M\cong \oplus P_{S^*}^{\oplus m_S}$. 
The assertion follows from Theorem \ref{duality}. 

If $\PP$ is an open pro-$p$ group of $\GG$ then $I|_{\PP}$ is an admissible injective object in $\Rep_k(\PP)$. Moreover, since 
$\PP$ is a pro-$p$ group the only irreducible representation of $\PP$ is the trivial one $\Eins$. And 
$m_{\Eins}=\dim_k \Hom_{\GG}(\Eins, I)=m$. The space of continuous functions $C(\PP, k)$ from $\PP$ to $k$ is an injective envelope
of $\Eins$ in $\Rep_k(\PP)$, Lemma \ref{injtrivial}. So $I|_{\PP}\cong C(\PP,k)^{\oplus m}$ hence $I^{\vee}|_{\PP}\cong k[[\PP]]^{\oplus m}$ as a $k[[\PP]]$-module.
This implies that $M\cong A[[\PP]]^{\oplus m}$. It follows from \cite{iw} Lemma 2.1, Corollary 2.2 that 
$(A[[\PP]])^d\cong C(\PP, L)$. 
\end{proof}

\begin{cor}\label{liftinjective} Let $I$ be an admissible injective object in $\Rep_k(\GG)$ then there exists an admissible unitary $L$-Banach space representation 
$(E, \|\centerdot\|)$, such that $E^0\otimes_A k\cong I$, where $E^0$ is the unit ball in $E$.  
\end{cor} 
\begin{proof} Let $P$ be a finitely generated projective object in $\CC(A[[\GG]])$ such that $P\otimes_A k\cong I^\vee$. Set $E:=P^d$, with the
supremum norm. Lemma \ref{dualityvee} implies the assertion.
\end{proof}

\begin{lem}\label{dcc} Let $E$ be an admissible $L$-Banach space representation of $\GG$, then any decreasing sequence of closed $\GG$-invariant 
subspaces becomes constant.
\end{lem}
\begin{proof} Suppose that $E_i\subseteq E$ is a closed $\GG$-invariant subspace. 
 Dually we have a surjection of $L[[\GG]]$-modules
$E^d \twoheadrightarrow E_i^d$, recall that  $\mfgfl{A[[\GG]]}_{\QQ}$ is equivalent to $\mfg{L[[\GG]]}$. Let $M_i$ denote the kernel, then we obtain an increasing sequence of $L[[\GG]]$-submodules of $E^d$, 
$M_1\subseteq M_2\subseteq \ldots E^d$. Since $E$ is admissible $E^d$ is finitely generated, and since $L[[\GG]]$ is noetherian  there exists $m$ such that $M_i=M_m$ for all $i\ge m$. Hence $E_i^d=E_m^d$ and so $E_i=E_m$ for all $i\ge m$.
\end{proof}

\section{Lifting $\Omega$}\label{sectionliftomega}
We assume throughout that $p\neq 2$. Let $G:=\GL_2(F)$, $Z$ the centre of $G$, $K:=\GL_2(\oF)$, 
$$I:=\begin{pmatrix} \oF^{\times}& \oF \\ \pF & \oF^{\times}\end{pmatrix}, \quad I_1:=\begin{pmatrix} 1+\pF & \oF \\ \pF & 1+\pF \end{pmatrix}.$$
Let $\KK_0$ be the $G$-normalizer of $K$, and $\KK_1$ be the $G$-normalizer of $I$, then $\KK_0=KZ$ and $\KK_1$ is generated as a group 
by $I$ and the element $\Pi:=\bigl (\begin{smallmatrix} 0 & 1 \\ \pif & 0\end{smallmatrix} \bigr )$. We fix a uniformizer $\pif$ of $F$, and consider 
as an element of $Z$, via $Z\cong F^{\times}$.

\begin{thm}\label{liftomega} Let $\Omega\in \Rep_k(G)$ be such that $\pif$ acts trivially, 
$\Omega|_K$ is an injective admissible object in $\Rep_k(K)$. 
 Then there exists a unitary admissible $L$-Banach space representation 
$(E,\|\centerdot\| )$  such that
$$E^0\otimes_A k\cong \Omega,$$
as $G$-representations, where $E^0$ denotes the unit ball in $E$, with respect to $\|\centerdot\|$.
\end{thm}
\begin{proof} Since $\Omega|_K$ is injective in $\Rep_k(K)$, $\Omega|_I$ is injective in $\Rep_k(I)$.
Since $\pif$ acts trivially on $\Omega$, we may consider $\Omega$ as a representation of  $\GG:=\KK_1/\pif^{\ZZ}$.
The index of $I$ in $\GG$ is $2$, and since $p\neq 2$ by assumption, we get that the maximal pro-$p$ subgroup 
of $\GG$ is contained in $I$. This implies that $\Omega|_{\GG}$ is an injective representation of $\GG$. Corollary
\ref{liftinjective} applied to $(\GG, \Omega)$, gives an admissible unitary $L$-Banach space representation $(E_1, \|\centerdot\|_1)$
of $\GG$ such that we have an isomorphism of $\GG$-representations $\iota_1:E_1^0\otimes_A k \cong \Omega$. We let $\pif$ act trivially 
on $E_1$, so that $\iota_1$ is $\KK_1$-equivariant. Corollary \ref{liftinjective} applied to $(K, \Omega)$  gives an admissible unitary 
$L$-Banach space representation $(E_0, \|\centerdot\|_0)$
of $K$ such that we have an isomorphism of $\GG$-representations $\iota_0:E_0^0\otimes_A k \cong \Omega$. We let $\pif$ act trivially 
on $E_0$, so that $\iota_0$ is $\KK_0$-equivariant. Now $\iota_0^{-1} \circ \iota_1$ induces an $ZI$-equivariant isomorphism 
$E_1^0\otimes_A k\cong E_0^0\otimes_A k$. It follows from Corollary \ref{liftisomorphism} that there exists a $ZI$-equivariant isometrical isomorphism: 
$$\phi: (E_1, \|\centerdot\|_1)\cong (E_0,\|\centerdot\|_0),$$
such that $\phi\otimes 1=  \iota_0^{-1} \circ \iota_1$. We may transport the action of $\KK_1$ on $(E_0,\|\centerdot\|_0)$, by setting 
$$g\centerdot v= \phi(g \phi^{-1}(v)), \quad \forall v\in E_0, \forall g\in \KK_1.$$    
If we restrict to $IZ=\KK_0\cap \KK_1$ the two actions coincide, since $\phi$ is $IZ$-equivariant. Since $G$ is an amalgam of $\KK_0$ and 
$\KK_1$ along $IZ$, the two actions glue to an action of $G$. So we get an $L$-Banach space representation of $G$ on $(E, \|\centerdot \|)$, 
which is unitary, since it is unitary for the actions of $\KK_0$ and $\KK_1$. By construction we obtain a $G$-equivariant isomorphism 
$E^0\otimes_A k\cong \Omega$. Instead of using the amalgamation, one could also argue formally as in \cite[Cor. 5.5.5]{coeff}.
\end{proof}  

We note that although the lifts $(E_i,\|\centerdot \|_i)$  are unique up to $\KK_i$-equivariant isometry, there is no 
unique way to choose $\phi$, so the Banach space representation $(E,\|\centerdot\|)$ is not canonical. Moreover, it is enough to 
assume that $\pif$ acts by a scalar on $\Omega$, since after twisting by an unramified character we may get to the situation of Theorem \ref{liftomega}.
The following is a Banach space analog of \cite[Cor. 9.11]{bp}.

\begin{prop}\label{converse} Let $(E_1, \|\centerdot\|_1)$ be an admissible unitary $L$-Banach space representation of $G$, 
such that $\pif$ acts trivially and $\|E_1\|_1\subseteq |L|$.
 Let $\sigma$ be the $K$-socle of $E^0_1\otimes_A k$, 
then there exists a unitary $L$-Banach space representation $(E, \|\centerdot\|)$ of $G$, such that the restriction of $E^0\otimes_A k$ to $K$ is 
an injective envelope of $\sigma$ in $\Rep_k(K)$ and a $G$-equivariant isometry $(E_1, \|\centerdot\|_1)\hookrightarrow (E,\|\centerdot\|)$.
\end{prop}
\begin{proof} Since $E_1$ is an admissible Banach space representation, $E^0_1\otimes_A k$ is an admissible smooth $k$-representation of $G$. 
By \cite[Cor. 9.11]{bp} there exists a $G$-equivariant embedding $\iota:E^0_1\otimes_A k\hookrightarrow \Omega$, where $\Omega$ is a smooth $k$-representation
of $G$, such that $\Omega|_K$ is an injective envelope of $\sigma$ in $\Rep_k(K)$. Let $(E, \|\centerdot\|)$ be a lift of $\Omega$, given by Theorem 
\ref{liftomega}.
For $i\in \{0, 1\}$ set $\GG_i:=\KK_i/\pif^{\ZZ}$ then dually, we have a diagram of $A[[\GG_i]]$-modules:
\begin{displaymath}
\xymatrix{ (E^0)^d\ar[d] & (E^0_1)^d\ar[d] \\ \Omega^{\vee}\ar[r]^-{\iota^{\vee}} & (E^0\otimes_A k)^{\vee}}
\end{displaymath}
Since $\Omega$ is injective, $\Omega^{\vee}$ is a projective $k[[\GG_i]]$-module and so $(E^0)^d$ is a projective $A[[\GG_i]]$-module, Proposition \ref{dproj}.
Hence there exists an $A[[\GG_i]]$-module homomorphism $\psi_i: (E_0)^d \rightarrow (E^0_1)^d$ making the diagram commute. Nakayama's Lemma implies that 
$\psi_i$ is surjective, see the proof of Proposition \ref{Misproj}. Since $\Omega|_I$ is injective in $\Rep_k(I)$, $(E^0)^d$ is a projective $A[[I]]$-module. 
Since $\psi_1$ and $\psi_2$ are also homomorphisms of $A[[I]]$-modules, Proposition \ref{2surj} gives us $\phi\in \Aut_{A[[I]]}((E^0)^d)$ such that 
$\psi_1=\psi_2\circ \phi$. We let $\pif$ act trivially everywhere. Dually we get a $\KK_0$-equivariant isometry $\psi^d_0: (E_1, \|\centerdot\|)\hookrightarrow 
(E, \|\centerdot\|)$, a $\KK_1$-equivariant isometry    
$\psi^d_1: (E_1, \|\centerdot\|)\hookrightarrow (E,\|\centerdot\|)$ and an $IZ=\KK_0\cap \KK_1$-equivariant 
isometrical isomorphism $\phi^d: (E, \|\centerdot\|)\cong (E,\|\centerdot\|)$, such that 
$\psi^d_1=\phi^d\circ \psi_2^d$. The data $E|_{\KK_0}$, $E|_{\KK_1}$, $\phi^d$ 
glues to a new representation of $G$, $(E', \|\centerdot\|)$ as in the proof of Theorem \ref{liftomega}, and by construction $E'|_K=E|_K$. Moreover, the map 
$\psi_0^d: E_1\rightarrow E'$ is a $G$-equivariant isometry, since by construction it is $\KK_0$ and $\KK_1$-equivariant isometry and these groups generate $G$.
\end{proof}   

\begin{cor}\label{cor3} Let $\kappa$ be an  irreducible smooth admissible $k$-representation of $G$, such that $\pif$ acts trivially. Then there exists 
an admissible topologically irreducible unitary $L$-Banach space representation $(E, \|\centerdot\|)$, such that $\Hom_G(\kappa, E^0\otimes_A k)\neq 0$.
\end{cor} 
\begin{proof} We may embed $\kappa\hookrightarrow \Omega$, where $\Omega$ is a smooth $k$-representation of $G$, such that $\Omega|_K$ is an injective envelope 
of $\kappa$ in $\Rep_k(K)$. Let $(E', \|\centerdot\|')$ be a unitary $L$-Banach space representation of $G$ lifting  $\Omega$ as in Theorem \ref{liftomega}. Since $E'$ 
is admissible, Lemma \ref{dcc} implies that $E'$ has an irreducible subobject $E$. Lemma \ref{redimq} gives an injection $E^0\otimes_A k\hookrightarrow \Omega$. Since 
$\Omega|_K$ is an injective envelope of $\kappa$, we get that $\kappa\cap  (E^0\otimes_A k)  \neq 0$. Since $\kappa$ is irreducible, it is contained in $E^0\otimes_A k$.
\end{proof}

\section{Admissible completions}\label{admissiblecompletions}
We briefly recall the theory of types for $\GL_2(F)$. Traditionally the 
smooth representations of $G$ are considered over the field of complex 
numbers, however, since the theory is algebraic in nature, we may consider it over any algebraically closed field of characteristic $0$. In the following 
we take the coefficients to be $\overline{L}$, the algebraic closure of $L$.
 Let $\mathfrak R$ be the category of smooth representations of $G$ on $\overline{L}$-vector spaces, 
then $\RR$ decomposes into a product of 
subcategories
$$\RR\cong \prod_{\mathfrak s\in \mathfrak B} \RR_{\mathfrak s},$$
where $\mathfrak B$  is the set of inertial equivalence classes of supercuspidal representations of the Levy subgroups of $G$, see \cite{ber}, \cite{bk}. Following
\cite[Def A.1.4.1]{henniart} we say that an irreducible smooth $\overline{L}$-representation $\tau$ is \textit{typical} for the Bernstein component $\RRs$, if for every 
irreducible object $\pi$ in $\RR$, $\Hom_K(\tau, \pi)\neq 0$ implies that $\pi$ lies in $\RRs$. We say that $\tau$ is a \textit{type} for $\RRs$ if it is typical and 
$\Hom_K(\tau, \pi)\neq 0$ for very irreducible object $\pi$ in $\RRs$. Given $\RRs$, there exists a type $\tau$, unique up to isomorphism, except when 
$\RRs$ contains $\chi\circ \det$. In this case, there are two typical representations $\theta\circ \det$ and $\St\otimes \theta\circ \det$, where
$\theta:=\chi|_{\oF^{\times}}$ and $\St$ is the lift to $K$ of the Steinberg representations of $\GL_2(\Fq)$, see \cite{henniart}. 
It follows from \cite{henniart} that the definition of a type here coincides with the one given in \cite[Def 4.1]{bk}. In particular, 
if $\tau$ is a type for $\RRs$ and $\pi$ is a smooth representation of $G$, with a $K$-invariant subspace $W$ isomorphic to $\tau$ then the 
subspace $\langle G\centerdot W\rangle$ of $\pi$, is an object of $\RRs$.          

Let $\Sigma$ be the set of $\Qp$-linear embeddings of fields $F\hookrightarrow L$. We choose $L$ to be 'large', 
so that $[F:\Qp]=|\Sigma|$. When needed we
will replace $L$ with some finite extension.

For us a $\Qp$-rational representation of 
$G$, is a representation $W$ of the form
$$\bigotimes_{\sigma\in \Sigma} (\Sym^{r_{\sigma}} L^2 \otimes \dt^{a_{\sigma}})^{\sigma},$$
where $r_{\sigma}$, $a_{\sigma}$ are integers, $r_{\sigma}\ge 0$, and an element $\begin{pmatrix} a & b\\ c & d\end{pmatrix}$ in $G$ acts on the $\sigma$-component  via 
$\begin{pmatrix} \sigma(a) & \sigma(b)\\ \sigma(c) & \sigma(d)\end{pmatrix}$, see \cite[\S2]{bsch} for a proper setting.  
Such $W$ are absolutely irreducible and remain absolutely irreducible when restricted to 
an open subgroup of $G$, (this is used implicitly in \cite[Lem 2.1]{bsch}, one may argue by Zariski density as in \cite{bsch}).  
The locally $\Qp$-rational representations in the title refer 
to the representations of the form $\pi\otimes_L W$, where $\pi$ is a smooth representation of $G$ on an $L$-vector space and $W$ is a 
$\Qp$-rational representation as above. If $\pi$ is absolutely irreducible then $\pi\otimes W$ is also absolutely irreducible, (once one knows 
that the restriction of $W$ to an open subgroup of $G$ remains absolutely irreducible, the argument in \cite{prasad} goes through).

\begin{lem}\label{uvw} Let $\GG$ be a group and $U$, $V$, $W$ representations of $\GG$ on $L$-vector spaces then
$$\Hom_{\GG}(U\otimes_L V, W)\cong \Hom_{\GG}(U, \Hom_L(V, W)),$$
where we consider $\Hom_L(V,W)$ as a representation of $\GG$, via $[g\centerdot\phi](v)= g(\phi(g^{-1} v))$.
\end{lem}
\begin{proof} We have an isomorphism of $L$-vector spaces $\Hom_L(U\otimes_L V, W)\cong \Hom_L(U, \Hom_L(V, W))$, $\phi\mapsto \phi'$, where 
$[\phi'(u)](v)= \phi(u\otimes v)$. Suppose that $\phi$ is $\GG$-equivariant then $[\phi'(gu)](v)=\phi(gu\otimes v)= g \phi(u\otimes g^{-1}v)= 
[g\centerdot\phi(u)](v)$. Hence, $\phi'$ is $\GG$-equivariant. Suppose that $\phi'$ is $\GG$-equivariant then $\phi(gu\otimes gv)=
[\phi'(gu)](gv)=[g\centerdot \phi'(u)](gv)= g \phi(u\otimes v)$.
\end{proof}

\begin{thm}\label{uff} Let $(E,\|\centerdot\|)$ be a unitary $L$-Banach space representation of $G$, such that $\|E\|\subseteq |L|$, $\pif$ acts trivially 
and the restriction 
of $E^0\otimes_A k$ to $K$ is an admissible injective
object in $\Rep_k(G)$. Let $\sigma$ be an irreducible $k$-representation of $K$, such that 
$\Hom_K(\sigma, E^0\otimes_A k)\neq 0$. Let $W$ be an irreducible 
$\Qp$-rational representation of $G$ tensored with a continuous character $\eta: G\rightarrow L^{\times}$. Let $\tau$ be an absolutely irreducible 
smooth representation of $K$, such that $\tau\otimes_L \overline{L}$ is typical for a Bernstein component $\RRs$. Choose a $K$-invariant lattice
$M$ in $\tau\otimes_L W$ and suppose that $\sigma$ occurs as a subquotient of $M\otimes_A k$. Then there exists a finite extension $L'$ of $L$, a smooth 
absolutely irreducible representation $\pi$ on an $L'$-vector space, such that $\pi_{\overline{L}}$ lies in $\RRs$, and a $G$-equivariant 
embedding 
$$\pi\otimes_{L'}W_ {L'}\hookrightarrow E_{L'}.$$
\end{thm}   
\begin{proof}
We note that  any $L$-Banach space topology on a finite dimensional $L$-vector space $V$ coincides with the finest locally convex one, \cite[Prop. 4.13]{nfa}. 
Hence any $L$-linear map from $V$  into any locally convex 
$L$-vector space is continuous, \cite[\S 5 C]{nfa}. Hence, we have
$$\Hom_K(\tau\otimes_L W, E)\cong \Hom_{L[[K]]}(E^d,(\tau\otimes_L W)^d).$$
It follows from Proposition \ref{dproj} and Corollary \ref{dim} that $\Hom_K(\tau\otimes W, E)$ is finite dimensional. Moreover, since $\sigma$ 
is a subquotient of $M\otimes_A k$,  Proposition \ref{dproj} and Corollary \ref{dim} imply that $\Hom_K(\tau\otimes_L W, E)$ is non-zero.
Since $\pif$ acts by a scalar on $W$, there exists a unique extension of $\tau$ to a representation $\tilde{\tau}$ of $KZ$ such that 
$\pif$ acts trivially on $\tilde{\tau}\otimes_L W$. Since $\pif$ act trivially on $E$ we have 
\begin{equation}\label{longiso}
\begin{split}
 \Hom_{K}(\tau\otimes_L W, E)&\cong \Hom_{KZ}(\tilde{\tau}\otimes_L W, E)\cong\Hom_{KZ}(\tilde{\tau}, \Hom_L(W, E))\\
&\cong \Hom_G(\cIndu{KZ}{G}{\tilde{\tau}}, \Hom_L(W, E)),
\end{split}
\end{equation}
where the second isomorphism is given by Lemma \ref{uvw}. Choose a non-zero $\phi\in \Hom_G(\cIndu{KZ}{G}{\tilde{\tau}}, \Hom_L(W, E))$ and let 
$\pi_1$ be the image of $\phi$. Since $\pi_1$ is a smooth representation, it will be contained in $\Hom^{sm}_L(W, E)$ consisting of 
$\phi\in \Hom_L(W,E)$
such that there exists an open subgroup $J$ of $G$, with $g \phi g^{-1}=\phi$, for all $g\in J$. Let 
$J$ be an open compact subgroup of $G$, then the 
subspace of $J$-invariants in $\Hom^{sm}_L(W, E)$ is equal to $\Hom_J(W,E)$, which is finite dimensional by Proposition 
\ref{dproj} and Corollary \ref{dim}.
Hence, $\Hom^{sm}_L(W,E)$ is an admissible representation of $G$. Since $\pi_1$ is finitely generated 
\cite[5.10]{vig2} implies that $\pi_1$ is of finite 
length as $L[G]$-module. Let $\pi_2$ be an irreducible $L[G]$-submodule of $\pi_1$. 
It follows from \cite[4.4]{vig2} that there exists a finite extension 
$L'$ of $L$, such that $\pi_2\otimes_L L'$ is a direct sum of absolutely irreducible representations. Let $\pi$ be an irreducible summand. Since
$\Hom_L(W,E)\otimes_L L'\cong \Hom_{L'}(W_{L'}, E_{L'})$, we get 
$$\Hom_{G}(\pi\otimes_{L'}W_{L'}, E_{L'})\cong \Hom_{G}(\pi, \Hom_{L'}(W_{L'}, E_{L'}))\neq 0.$$ 
Since $\pi$ is irreducible and $W$ is $\Qp$-rational tensored with a character, 
the representation $\pi\otimes_{L'} W_{L'}$ is irreducible, and hence 
any non-zero homomorphism is an injection. 

It remains to show that $\pi_{\overline{L}}$ lies in $\RRs$. We know that $\pi_{\overline{L}}$ is a subobject of 
$\pi_1\otimes_L \overline{L}$, which 
is generated as a $G$-representation by a subspace isomorphic to $\tau_{\overline{L}}$. If $\tau$ is a type, 
then we are done. If $\tau$ is not a type, then 
$\tau\cong \chi\circ \det$ or $\tau\cong \St \otimes \chi\circ \det$, for some smooth character $\chi: \oF^{\times}\rightarrow L^{\times}$, where $\St$
denotes a lift to $K$ of the Steinberg representation of $K/K_1\cong \GL_2(\mathbb F_q)$. By twisting we may assume $\chi$ to be trivial. 
Then the trivial 
representation of $I$ is a type for $\RRs$. Since $\Hom_I(\Eins, \tau)\neq 0$, we get that  
$\pi_1\otimes_L \overline{L}$ is generated by a subspace 
isomorphic to the trivial representation of $I$, and hence $\pi_{\overline{L}}$ lies in $\RRs$.      
\end{proof}

We also give a variant of Theorem \ref{uff} when $\tau$ is the trivial representation of $K$. 

\begin{lem}\label{latticefin} Let $V$ be an absolutely irreducible  representation of $G$ on a finite dimensional $L$-vector space. Suppose that
$G$ stabilizes a lattice in $V$ then $V$ is a character.
\end{lem}
\begin{proof} If $G$ stabilizes a lattice in $V$, then we obtain a group homomorphism $\rho:G\rightarrow \GL_d(A)$, where $d=\dim_L V$.
It follows from \cite[Thm. 8.4]{lang} that $\SL_2(F)$ does not have a non-trivial quotient of finite order. 
Hence, $\rho(\SL_2(F))$ is contained in $1+\varpi_L^n M_d(A)$, for 
all $n\ge 1$. Since the intersection of these groups is trivial, $\rho(\SL_2(F))=1$. Hence, $\rho(G)$ is abelian and so for $g\in G$, the map 
$v\mapsto gv$ lies in $\End_G(V)$. Since $V$ is finite dimensional and absolutely irreducible, Schur's lemma gives 
 $\End_G(V)=L$, and hence $G$ acts by a character.
\end{proof}

\begin{cor}\label{principalseries} Let $E$, $\sigma$, $W$ be as in Theorem \ref{uff}. Let $M$ be a $K$-stable lattice in $W$ and suppose that 
$\sigma$ is a subquotient of 
$M\otimes_A k$. Assume that $W$ is not one dimensional then there exists a finite extension $L'$ of $L$ and 
an unramified principal series representation 
$\pi$ of $G$ on an $L'$-vector space and a $G$-equivariant injection $\pi\otimes_{L'} W_{L'}\hookrightarrow E_{L'}$.
\end{cor}
\begin{proof} We proceed as in the proof of Theorem \ref{uff} with $\tau=\Eins$. Let $\tilde{\Eins}$ denote an unramified character such that
$\pif$ acts trivially on $\tilde{\Eins}\otimes W$. The equation \eqref{longiso} implies that  
$N:=\Hom_G(\cIndu{KZ}{G}{\tilde{\Eins}}, \Hom_L(W, E))$
is a non-zero finite dimensional vector space. It is also naturally a module for the Hecke algebra $\HH:=\End_G(\cIndu{KZ}{G}{\tilde{\Eins}})$. 
The Hecke algebra $\HH$ is isomorphic to $L[T]$ the polynomial ring in one variable. So if we choose some non-zero 
$\phi\in N$ then the $\HH$-module generated by $\phi$ is isomorphic to $L[T]/(P)$, for some
polynomial $P$. Hence if we let $\pi_1$ be the image of $\phi$, then $\pi_1$ is isomorphic to 
$$\cIndu{KZ}{G}{\tilde{\Eins}}\otimes_{L[T]} L[T]/(P)\cong \cIndu{KZ}{G}{\tilde{\Eins}}/(P).$$
Let $L'$ be the splitting field of $P$, and $a$ a root of $P$. If we set 
$\pi:= \frac{\cIndu{KZ}{G}{\tilde{\Eins}_{L'}}}{(T-a)}$, then $\pi$ will be isomorphic 
to a subobject of $\pi_1\otimes_L L'$, so we have 
$$\Hom_G(\pi\otimes_{L'} W_{L'}, E_{L'})\cong \Hom_{G}(\pi, \Hom_{L'}(W_{L'}, E_{L'}))\neq 0.$$ 
Now $\pi$ is an unramified  principal series representation. If $\pi$ is irreducible we are done. 
If $\pi$ is reducible then 
it is a non-split extension 
$0\rightarrow \St\otimes \chi\rightarrow \pi\rightarrow \chi\rightarrow 0$, for some character $\chi:G\rightarrow L^{\times}$, and 
$\St$ denotes the Steinberg representation of $G$. The sequence  
$0\rightarrow \St\otimes W\chi\rightarrow \pi\otimes W\rightarrow W\chi\rightarrow 0$
is also non-split, otherwise by tensoring with $W^*$ and 
taking smooth vectors we would obtain the splitting of the original sequence. So if there exists 
 non-zero $\psi\in \Hom_G(\pi\otimes_{L'} W_{L'}, E_{L'})$, such that $\Ker \psi\neq 0$ then the image of $\psi$ is isomorphic to $W\chi$. But then $W\chi\cap E^0$
would be a $G$-invariant lattice in $W\chi$, which would contradict Lemma \ref{latticefin}.
\end{proof} 

Since all the $L$-Banach spaces below arise from the constructions of 
Theorem \ref{uff} and Corollary \ref{principalseries} we always assume that $\|E\|\subseteq |L|$.  

\begin{cor}\label{cor1} Let $\tau$ be a smooth absolutely irreducible $L$-representation of $K$, 
which is typical for Bernstein component  $\RRs$. Let  $W$ be a 
$\Qp$-rational representation of $G$, twisted by a continuous character. Let $M$ be a $K$-invariant lattice in $\tau\otimes W$. 
Let $\kappa$ be an absolutely irreducible smooth admissible 
$k$-representation of $G$, such that $\pif$ acts trivially. Suppose 
that there exists an  irreducible 
$k$-representation $\sigma$ of $K$, such that 
\begin{itemize}
\item[(1)] $\Hom_K(\sigma, \kappa)\neq 0$;
\item[(2)] $\sigma$ occurs as a subquotient of $M\otimes_A k$.
\end{itemize}
Then there exists a finite extension $L'$ of $L$, 
an absolutely irreducible smooth $L'$-representation $\pi$ of $G$ in $\RRs$, and an admissible unitary 
$L'$-Banach space representation $(E,\|\centerdot\|)$ of $G$, such that the following hold:
\begin{itemize}
\item[(i)] $\pi\otimes_{L'} W_{L'}$ is a dense $G$-invariant subspace of $E$;
\item[(ii)] $\Hom_G(\kappa_{k'}, E^0\otimes_{A'} k')\neq 0$.
\end{itemize} 
\end{cor}
\begin{proof} Since $\kappa$ is admissible, by \cite[Cor.9.11]{bp} 
there exists a $G$-equivariant embedding $\kappa\hookrightarrow \Omega$, where 
$\Omega$ is a smooth representation of $G$, such that $\Omega|_K$ is an injective envelope of 
$\soc_K \kappa$ in $\Rep_k (K)$. Let $(E', \|\centerdot\|)$ 
be a lift of $\Omega$ as in Theorem \ref{liftomega}. Then by 
Theorem \ref{uff} there exists a finite extension $L'$ of $L$ and an absolutely irreducible 
smooth $L'$-representation $\pi$ in $\RRs$, and 
a $G$-equivariant embedding $\pi\otimes_{L'}W_{L'}\hookrightarrow E'_{L'}$. Let $E$ be the closure of 
$\pi\otimes_{L'} W_{L'}$ in $E'$ and $E^0$ the unit ball in $E$ with respect to $\|\centerdot\|$. Since $E'$ is admissible, so is $E$.
Lemma \ref{redimq} gives a $G$-equivariant injection
$E^0\otimes_{A'} k'\hookrightarrow \Omega_{k'}$. It follows from Lemma \ref{ratsoc} that 
$\soc_K \Omega_{k'}\cong (\soc_K \Omega) \otimes_k k'\cong (\soc_K \kappa)\otimes_k k'$, hence 
$\kappa_{k'}\cap (E^0\otimes_{A'} k')\neq 0$. Since $\kappa$ is 
absolutely irreducible we get that $\kappa_{k'}$ is contained in $E^0\otimes_{A'} k'$.    
\end{proof}

\begin{cor}\label{cor2} Let  $W$ be 
$\Qp$-rational representation of $G$, twisted by a continuous character. Let $M$ be a $K$-invariant lattice in $W$. 
Let $\kappa$ be an absolutely irreducible smooth admissible $k$-representation of $G$, 
such that $\pif$ acts trivially. Suppose that there exists an  irreducible 
$k$-representation $\sigma$ of $K$, such that 
\begin{itemize}
\item[(1)] $\Hom_K(\sigma, \kappa)\neq 0$;
\item[(2)] $\sigma$ occurs as a subquotient of $M\otimes_A k$.
\end{itemize}
Assume that either $\kappa$ is not finite dimensional or $W$ is not a character. Then there exists a finite extension $L'$ of $L$, an 
unramified  smooth principal series $L'$-representation $\pi$ of $G$, and an admissible unitary 
$L'$-Banach space representation $(E,\|\centerdot\|)$ of $G$, such that the following hold:
\begin{itemize}
\item[(i)] $\pi\otimes_{L'} W_{L'}$ is a dense $G$-equivariant subspace of $E$;
\item[(ii)] $\Hom_G(\kappa_{k'}, E^0\otimes_{A'} k')\neq 0$.
\end{itemize} 
\end{cor}
\begin{proof} The proof is the same as of Corollary \ref{cor1}, using Corollary \ref{principalseries} instead of Theorem \ref{uff}.
\end{proof}

\begin{remar}\label{smoothtoalgebraic} We note that any irreducible smooth $k$-representation $\sigma$ of $K$ is of the form  
$$\bigotimes_{\tau:k_F\hookrightarrow k} (\Sym^{r_{\tau}} k^2 \otimes \dt^{a_{\tau}})^{\tau}$$
with $0\le r_{\tau}\le p-1$ and $0\le a_{\tau}<p-1$. Hence, we may lift $\sigma$ to a $\Qp$-rational representation 
$$W:=\bigotimes_{\tilde{\tau}:F\hookrightarrow L} (\Sym^{r_{\tau}} L^2 \otimes \dt^{a_{\tau}})^{\tilde{\tau}},$$
where for each $\tau$ we fix $\tilde{\tau}:F\hookrightarrow L$, inducing $\tau$ on the residue fields.
\end{remar} 

We also note that although we know that the completion $E$ in Corollaries \ref{cor1}, \ref{cor2} is admissible, we do not know in general 
whether it is of finite length as a (topological) representation of $G$.

\begin{lem}\label{trivialll} Let $\pi$ be a smooth $L$-representation of $G$; $\tau$ a smooth irreducible representation of $K$, such that $\Hom_K(\tau,\pi)\neq 0$; 
$W$ a $\Qp$-rational representation of $G$, twisted by a continuous character. Suppose that $(E,\|\centerdot\|)$ is a unitary $L$-Banach space representation
of $G$, which contains $\pi\otimes_L W$ as a dense $G$-invariant subspace. Choose a $K$-invariant lattice $M$ in $\tau\otimes_L W$, then there exists 
an irreducible subquotient $\sigma$ of $M\otimes_A k$ such that $\Hom_K(\sigma, E^0\otimes_A k)\neq 0$.
\end{lem}
\begin{proof} Let $M':= (\tau\otimes_L W)\cap E^0$, then Lemma  \ref{redimq} gives an injection $M'\otimes_A k\hookrightarrow E^0\otimes_A k$. In particular,
there exists some irreducible $k$-representation $\sigma$ of $K$, such that $\Hom_K(\sigma, M'\otimes_A k)\neq 0$ and  $\Hom_K(\sigma, E^0\otimes_A k)\neq 0$.
Since $\tau\otimes_L W$ is finite dimensional we have $(M'\otimes_A k)^{ss}=(M\otimes_A k)^{ss}$. 
\end{proof}

Lemma \ref{trivialll} together with  Proposition \ref{converse} 
shows that any admissible completion of $\pi\otimes W$ arises from our construction.

\begin{lem}\label{reform} 
Let $\chi_1$, $\chi_2: F^{\times}\rightarrow L^{\times}$ be smooth characters and let $\pi:=\Indu{B}{G}{\chi_1\otimes \chi_2|\centerdot|^{-1}}$ be a 
smooth $L$-representation of $G$. For each $\sigma\in \Sigma$, let $r_{\sigma}\ge 0$ be an integer, set 
$W:=\otimes_{\sigma\in \Sigma} (\Sym^{r_{\sigma}} L^2)^{\sigma}$, 
and let  $\eta: F^{\times} \rightarrow L^{\times}$ be a continuous character. Suppose that $\pi\otimes W\otimes \eta\circ \det$ admits 
a unitary completion and set $\lambda_1:=\chi_1(\pif)^{-1}$ and 
$\lambda_2:= \chi_2(\pif)^{-1}$, then the following hold:
\begin{itemize}
\item[(i)] $(\val(\lambda_1)-\val(\eta(\pif))+(\val(\lambda_2)-\val(\eta(\pif)))=(1+ \sum_{\sigma\in \Sigma} r_{\sigma})/e$;
\item[(ii)] $\val(\lambda_2)-\val(\eta(\pif))\ge 0$, $\val(\lambda_1)-\val(\eta(\pif))\ge 0$;
\end{itemize}
where $e:=e(F|\Qp)$ is the ramification index. Moreover, if $\varpi\in Z$ acts trivially on $\pi\otimes W \otimes \eta \circ \det$ then 
$$\lambda_1\lambda_2=p^{1/e}\eta(\pif)^2 \prod_{\sigma\in \Sigma} \sigma(\pif)^{r_{\sigma}}.$$
\end{lem}
\begin{proof} We have $\pi\otimes W\otimes \eta\circ \det\cong (\pi\otimes |\eta|^{-1}\circ \det)\otimes W\otimes (\eta |\eta|)\circ \det$. The character $\eta|\eta|$ is 
unitary, and if we let $\chi_1':=\chi_1 |\eta|^{-1}$ and $\chi_2':=\chi_2 |\eta|^{-1}$ then the characters $\chi_1'$ and $\chi_2'$ are smooth and  
$\pi\otimes |\eta|^{-1}\circ \det\cong \Indu{B}{G}{\chi_1'\otimes \chi_2'|\centerdot|^{-1}}$. Since $\val(\chi_i'(\pif))= \val(\chi_i(\pif))+\val(\eta(\pif))$, for 
$i\in \{1,2\}$, we may assume that $\eta$ is the trivial character.

Note that $\pif$ acts on $\pi\otimes_L W$ by a scalar 
$\lambda_1^{-1}(p^{1/e}\lambda_2^{-1}) \prod_{\sigma\in \Sigma} \sigma(\pif)^{r_{\sigma}}$. If $\pi\otimes W$ admits a unitary completion then the 
central character of 
$\pi\otimes W$ has to be unitary, which is equivalent to $\pif$ acting by a scalar in $A^{\times}$. This implies 
$$\val(\lambda_1)+\val(\lambda_2)= (1+\sum_{\sigma\in \Sigma} r_{\sigma})/e.$$ 
This gives the first and also the last assertion.  

Let $\varphi\in \pi$ be the function such that $\supp \varphi= B sI= B s (I_1 \cap U)$ and
$\varphi(su)=1$, for all $u\in I_1\cap U$. Then $t\varphi$ is the unique function in $\pi$ with the support  
$ (\supp \varphi) t^{-1}= B s (K_1\cap U)$ and $[t\varphi](su)=\chi_2(\pif)|\pif|^{-1}$, for all $u\in K_1\cap U$. Hence, we have 
\begin{equation}\label{sumvarphi}
\sum_{\lambda \in \oF/\pF} \begin{pmatrix} 1 & \lambda \\ 0 & 1\end{pmatrix} t \varphi = \lambda_2^{-1} p^{1/e} \varphi
\end{equation} 
Let $X^{\mathbf{r}}\in W$, $X^{\mathbf{r}}= \otimes_{\sigma\in \Sigma} X^{r_{\sigma}}$. Then $U$ acts trivially on $X^{\mathbf{r}}$ and $t X^{\mathbf{r}}=
(\prod_{\sigma\in \Sigma} \sigma(\pif)^{r_{\sigma}}) X^{\mathbf r}$. Hence, \eqref{sumvarphi} gives us 
\begin{equation}\label{sumvarphiX}
 \sum_{\lambda \in \oF/\pF} \begin{pmatrix} 1 & \lambda \\ 0 & 1\end{pmatrix} t (\varphi\otimes X^{\mathbf r}) = 
(\lambda_2^{-1} p^{1/e} \prod_{\sigma\in \Sigma}\sigma(\pif)^{r_{\sigma}})
\varphi\otimes X^{\mathbf{r}}.
\end{equation}
Suppose that  $\pi\otimes W$ admits a unitary completion. Since the action
of $G$ is unitary, the triangle inequality applied to \eqref{sumvarphiX} gives $\lambda_2^{-1} p^{1/e} \prod_{\sigma\in \Sigma}\sigma(\pif)^{r_{\sigma}}\in A$. Hence, 
$\val(\lambda_2)\le (1 +\sum_{\sigma\in \Sigma} r_{\sigma})/e$.  If $\chi_1=\chi_2|\centerdot|$ then 
$\val(\lambda_1)< \val(\lambda_2)$ and so (ii) follows from (i). If $\chi_1=\chi_2|\centerdot|^{-1}$ then 
$W\otimes \chi_1\circ\det$ is a subrepresentation of $\pi\otimes W$. Hence, we obtain a $G$-invariant norm 
on $W\otimes \chi_1\circ\det$. Lemma \ref{latticefin} implies that $r_{\sigma}=0$ for all $\sigma\in \Sigma$. Part (i) gives 
$\val (\lambda_1)+\val(\lambda_2)=1/e$ and so $\val(\lambda_2)=1/e$ and $\val(\lambda_1)=0$.

Suppose that $\chi_1\neq \chi_2|\centerdot|^{\pm 1}$
then $\pi$ is irreducible and the intertwining operator induces an isomorphism 
$\pi\cong\Indu{B}{G}{\chi_2\otimes \chi_1|\centerdot|^{-1}}$, see \cite[Thm 4.5.3]{bump}. Hence, we also obtain
$\val(\lambda_1)\le (1 +\sum_{\sigma\in \Sigma} r_{\sigma})/e$, which implies (ii).  
\end{proof}

Let $\chi_1$, $\chi_2: F^{\times}\rightarrow L^{\times}$ be smooth characters and let $\pi:=\Indu{B}{G}{\chi_1\otimes \chi_2|\centerdot|^{-1}}$ be a 
smooth $L$-representation of $G$. Set $\lambda_1:=\chi_1(\pif)^{-1}$ and $\lambda_2=\chi_2(\pif)^{-1}$.  For each $\sigma\in \Sigma$, let $r_{\sigma}\ge 0$ be an integer, set 
$W:=\otimes_{\sigma\in \Sigma} (\Sym^{r_{\sigma}} L^2)^{\sigma}$, 
and let  $\eta: F^{\times} \rightarrow L^{\times}$ be a continuous character. Suppose that $\pi\otimes W\otimes \eta\circ \det$ is a dense $G$-invariant 
subspace in a unitary $L$-Banach space representation $E$ of $G$. Let $\theta: F^{\times} \rightarrow L^{\times}$ be the character $\theta(x):=
\prod_{\sigma\in \Sigma} \sigma(x)^{r_{\sigma}}$. Assume that $\val(\lambda_1)=\val(\eta (\pif))$, then it follows from Lemma \ref{reform} that 
the characters $\chi_1 \eta$ and $\chi_2|\centerdot|^{-1} \eta \theta$ are integral. Let 
$$\psi_1:=\chi_1\eta \pmod{1+\MM}, \quad \psi_2:=\chi_2 |\centerdot|^{-1}\eta \theta \pmod{1+\MM}.$$

\begin{lem}\label{ordinary} Assume that we are in the situation as above. Let $\|\centerdot \|$ be a $G$-invariant norm defining the topology on $E$ and let $E^0$ 
be the unit ball with respect to $\|\centerdot\|$ 
then $\Hom_G(\Indu{B}{G}{\psi_1\otimes\psi_2}, E^0\otimes_A k)\neq 0.$
\end{lem}
\begin{proof} We note that  the assertion is true for $\pi$ if and only if it is true for $\pi\otimes \xi\circ \det$, for some  $\xi: F^{\times}\rightarrow A^{\times}$ an
unramified character. Hence, we may assume that $\eta$ is trivial, and $\pif$ acts trivially (possibly after replacing  $L$ with a quadratic extension). Let $\phi\in \pi$ 
and $X^{\mathbf r}\in W$ be as in the proof of Lemma \ref{reform}. We may assume that $\|\phi\otimes X^{\mathbf r}\|=1$. Let $v$ be the image of $\phi\otimes X^{\mathbf r}$
in $E^0\otimes_A k$, then $v\neq 0$, and $I_1\cap B$ acts trivially on $v$, 
$$\begin{pmatrix} [\lambda] & 0 \\ 0 & [\mu]\end{pmatrix}  v= \psi_2([\lambda]) \psi_1([\mu]) v, \quad \forall \lambda, \mu\in k_F.$$ 
Set $u:=\lambda_2^{-1}p^{1/e} \prod_{\sigma\in \Sigma} \sigma(\pif)^{r_{\sigma}}$, then Lemma \ref{reform} implies that $u$ is a unit in $A$. Let $\bar{u}$
be the image of $u$ in $k$. Then \eqref{sumvarphiX} reduces to:
\begin{equation}\label{sumv}
 \sum_{\lambda \in \oF/\pF} \begin{pmatrix} 1 & \lambda \\ 0 & 1\end{pmatrix} t v =  \bar{u} v.
\end{equation}
Since $\bar{u}\neq 0$ we  are in the situation described in the proof of \cite[Thm. 5.4]{borel}, ($v$ corresponds to $\phi_2$). The argument there gives:
\begin{itemize}
\item[(i)] $v$ is fixed by $I_1$;
\item[(ii)]$\sigma:=\langle K\centerdot v\rangle$ is an irreducible representation of $K$, and if $\psi_1|_{\oF^{\times}}=\psi_2|_{\oF^{\times}}$ 
then $\sigma\cong \St\otimes \psi_1\circ \det$;
\end{itemize} 
Now \eqref{sumv} and \cite[Lem 3.1]{borel} together with \cite[Thm. 30]{bl} implies that the map 
$\cIndu{KZ}{G}{\sigma}\rightarrow \langle G\centerdot v\rangle$ factors through 
$\Indu{B}{G}{\xi_1\otimes \xi_2}$, where  
$\xi_1|_{\oF^{\times}}=\psi_1$, $\xi_2|_{\oF^{\times}}=\psi_2$, $\xi_2(\pif)=\xi_1(\pif^{-1})=\bar{u}$, 
This gives the result.
\end{proof}
One may call the situation of Lemma \ref{ordinary} the ordinary case. Lemma \ref{ordinary} shows that most of the time the completions we obtain 
in Corollaries \ref{cor1} and \ref{cor2} are not ordinary.

\section{The case $F=\Qp$}\label{FequalsQp}
We assume that $F=\Qp$ and $p>2$, and we study in more detail the consequences of Corollaries \ref{cor1}, \ref{cor2}.  Barthel-Livn\'e have shown in \cite{bl} 
that smooth irreducible $\bar{k}$-representations of $G$, with the central character fall into four disjoint classes:
\begin{itemize}
\item[(1)] $\chi\circ \det$;
\item[(2)] $\Sp\otimes \chi\circ \det$;
\item[(3)] $\Indu{B}{G}{\chi_1\otimes \chi_2}$, $\chi_1\neq \chi_2$;
\item[(4)] supersingular;
\end{itemize} 
where $\Sp$ is the Steinberg representation defined by the exact sequence $0\rightarrow \Eins\rightarrow \Indu{B}{G}{\Eins}\rightarrow \Sp\rightarrow 0$.
Breuil in \cite{breuil1} has classified the supersingular representations. We recall the classification. Fix an integer $0\le r \le p-1$, then 
the representation $\Sym^r \bar{k}^2$ of $K$ is irreducible. We put the action of $KZ$ on  $\Sym^r \bar{k}^2$ by making $p$ act trivially. 
Let $\HH$ be the Hecke algebra, $\HH:=\End_G(\cIndu{KZ}{G}{\Sym^r \bar{k}^2})$. Proposition 8 of \cite{bl} asserts that
as a $\bar{k}$-algebra $\HH$ is isomorphic to a polynomial ring in one variable $\bar{k}[T]$, where $T\in \HH$ is an endomorphism defined in \cite[\S3]{bl}. 
Moreover, $\cIndu{KZ}{G}{\Sym^r \bar{k}^2}$ is a free $\HH$-module, \cite[Thm.19]{bl}. Define, 
$$\kappa(r):=  \frac{ \cIndu{KZ}{G}{\Sym^r \bar{k}^2}}{(T)},$$ 
and if $\eta:\Qp^{\times}\rightarrow \bar{k}^{\times}$ is a smooth character, then set $\kappa(r,\eta):=\kappa(r)\otimes\eta\circ \det$. Breuil has shown  
\cite[Thm.1.1]{breuil1} that the representations $\kappa(r,\eta)$ are irreducible and any irreducible supersingular representation of $G$ is isomorphic to 
$\kappa(r,\eta)$, for some $0\le r\le p-1$ and $\eta$. All the isomorphism between supersingular representations corresponding 
to different $r$ and $\eta$ are given by 
\begin{equation}\label{intertwine}
\kappa(r,\eta)\cong \kappa(r,\eta\mu_{-1})\cong \kappa(p-1-r,\eta\omega^{r})\cong \kappa(p-1-r,\eta\omega^{r}\mu_{-1})
\end{equation}
see \cite[Thm. 1.3]{breuil1}, where $\omega:\Qp^{\times}\rightarrow \bar{k}^{\times}$ is a character given by $\omega(p)=1$ and $\omega|_{\Zp^{\times}}$ is the natural
map $\Zp^{\times}\rightarrow \Fp^{\times}\rightarrow \bar{k}^{\times}$, and given $\lambda\in \bar{k}$, we denote by 
$\mu_{\lambda}: \Qp^{\times}\rightarrow \bar{k}^{\times}$ the unramified character $x\mapsto \lambda^{\val(x)}$. 

\begin{lem}\label{rational} Every smooth irreducible $\bar{k}$-representation of $G$ with a central character can be realized over a finite extension of $k$.
\end{lem}
\begin{proof} Since $\Qp^{\times}$ is finitely generated as a topological group, given a smooth character $\eta: \Qp^{\times}\rightarrow \bar{k}^{\times}$, 
the image $\eta(\Qp^{\times})$ lies in a finite extension of $k$. Hence, all the principal series representations in (3) can be realized over a 
finite extension and if $\kappa$ can be realized over a finite extension, then so can 
$\kappa\otimes \eta\circ \det$. It is clear that the trivial representation can be realized over $\Fp$, and then we may realize $\Sp$ as the quotient
$0\rightarrow \Fp\rightarrow \Indu{B}{G}{\Fp} \rightarrow \Sp\rightarrow 0$. Moreover, we may realize $\kappa(r)$ over $\Fp$ as 
$  \frac{ \cIndu{KZ}{G}{\Sym^r \Fp^2}}{(T)}$, as $T$ in this case is defined over $\Fp$, \cite[Prop 8]{bl}.
\end{proof}

\begin{lem}\label{socleK} Let $\sigma:=\Sym^{r} \bar{k}^2\otimes \det^a$, with $0\le r\le p-1$, $0\le r<p-1$. Let $\kappa$ be an irreducible smooth 
$\bar{k}$-representation of $G$, such that $p\in Z$ acts trivially on $\kappa$. Then $\Hom_K(\sigma, \kappa)\neq 0$ if and only if one of the following holds:
\begin{itemize} 
\item[(i)] $r=0$ and $\kappa$ is isomorphic to one of the following: $\kappa(0, \omega^a)$, $(\mu_{\pm 1}\omega^a )\circ \det$ or 
$\Indu{B}{G}{\mu_{\lambda^{-1}}\omega^a\otimes \mu_{\lambda}\omega^a}$, for $\lambda\in \bar{k}^{\times}\setminus\{\pm 1\}$;
\item[(ii)] $r=p-1$ and $\kappa$ is isomorphic to one of the following: $\kappa(p-1, \omega^a)$,  $\Sp\otimes (\mu_{\pm 1}\omega^a) \circ \det$ or 
$\Indu{B}{G}{\mu_{\lambda^{-1}}\omega^a\otimes \mu_{\lambda}\omega^a}$, for $\lambda\in \bar{k}^{\times}\setminus\{\pm 1\}$;
\item[(iii)] $0<r<p-1$ and $\kappa$ is isomorphic to one of the following: 
$\kappa(r,\omega^a)$ or  $\kappa\cong \Indu{B}{G}{\mu_{\lambda^{-1}}\omega^a\otimes \mu_{\lambda}\omega^{a+r}}$, for $\lambda\in \bar{k}^{\times}$;  
\end{itemize}
\end{lem}
\begin{proof} This is well known. 
\end{proof}

\begin{remar} Since $k_L$ contains $\mathbb F_{p}$  every irreducible $k_L$-representation of $K$ is absolutely irreducible.
Lemma \ref{ratsoc} allows us to use Lemma \ref{socleK} with $k_L$ instead of $\bar{k}$.
\end{remar} 

\begin{lem}\label{addsoc} Let $\psi_1$, $\psi_2:\Qp^{\times}\rightarrow k_L^{\times}$ be smooth characters. 
Suppose that $\Indu{B}{G}{\psi\otimes\psi_2}$ has an irreducible subquotient $\kappa$ with $\soc_K \kappa\cong 
\Sym^r k_L^2 \otimes \det^a$, with $0\le r\le p-1$, $0\le a<p-1$, then
$(\psi_1|_{\Zp^{\times}}, \psi_2|_{\Zp^{\times}})=(\omega^a, \omega^{a+r})$.
\end{lem}
\begin{proof} If $\Indu{B}{G}{\psi_1\otimes\psi_2}$ is irreducible then the assertion follows from Lemma \ref{socleK} and 
\cite[Thm 34 (2)]{bl}. If  $\Indu{B}{G}{\psi_1\otimes\psi_2}$ is reducible, then $\psi_1=\psi_2$, and 
$(\Indu{B}{G}{\psi_1\otimes\psi_2})^{ss}\cong \psi_1\circ \det \oplus \Sp\otimes\psi_1\circ \det$,
\cite[Thm 30 (1)]{bl}. Hence, $r=p-1$ or $r=0$ and so $(\psi_1|_{\Zp^{\times}}, \psi_2|_{\Zp^{\times}})=(\omega^a, \omega^a)$.
\end{proof}

Let $\theta_1, \theta_2: \Zp^{\times}\rightarrow L^{\times}$ be smooth characters. If $\theta_1=\theta_2$ we set $c(\theta_1, \theta_2):=0$, 
$J_c:=K$ and $\tau(\theta_1, \theta_2):=\theta_1\circ \det$. If $\theta_1\neq \theta_2$ let $c(\theta, \theta_2)\ge 1$ be the smallest integer $c$, 
such that $\theta_1\theta_2^{-1}$ is trivial on $1+p^c \Zp$.  We  set 
$$J_c:=\begin{pmatrix} \Zp^{\times} & \Zp\\ p^c \Zp & \Zp^{\times}\end{pmatrix},$$
and we consider $\theta_1\otimes \theta_2$ as a character of $J_{c}$, 
by 
$$\theta_1\otimes \theta_2( \begin{pmatrix} a & b \\ c & d\end{pmatrix}):=\theta_1(a) \theta_2(d).$$ 
Set $\tau(\theta_1,\theta_2):=\Indu{J_{c}}{K}{\theta_1\otimes\theta_2}$. If $\theta_1\neq \theta_2$ then $\tau(\theta_1, \theta_2)$ is a type for the 
Bernstein component containing representations $\Indu{B}{G}{\chi_1\otimes\chi_2}$ with $\chi_1|_{\Zp^{\times}}=\theta_1$ and $\chi_2|_{\Zp^{\times}}=\theta_2$, 
see \cite[A.2.2]{henniart}. If $\theta_1=\theta_2$ then $\tau(\theta_1, \theta_2)$ is typical for the Bernstein component containing $\chi\circ \det$, with 
$\chi|_{\Zp^{\times}}=\theta_1=\theta_2$.

\begin{thm}\label{Qp} Let $\sigma:=\Sym^r k_L^2 \otimes \det^a$ with $0\le r\le p-1$ and $0\le a <p-1$. And let 
$\kappa$ be an 
absolutely irreducible smooth $k_L$-representation with a central character, such that $\Hom_K(\sigma, \kappa)\neq 0$ and $p\in Z$ acts trivially on $\kappa$. 
Fix an integer $k\ge 2$ and smooth characters 
$\theta_1, \theta_2:\Zp^{\times}\rightarrow L^{\times}$. Let $M$ be a $K$-stable lattice in $\tau(\theta_1, \theta_2)\otimes\Sym^{k-2} L^2$.
Suppose that $\sigma$ is a subquotient of $(M\otimes_A k_L)^{ss}$ then there exists a finite extension $L'$ of $L$,  smooth characters 
$\chi_1, \chi_2:\Qp^{\times}\rightarrow (L')^{\times}$ and an admissible unitary completion $E$ of $(\Indu{B}{G}{\chi_1\otimes \chi_2|\centerdot|^{-1}})\otimes 
\Sym^{k-2}(L')^2$ such that the following hold:
\begin{itemize}
\item[(1)] $\chi_1\neq \chi_2$;
\item[(2)] $\chi_1|_{\Zp^{\times}}=\theta_1$, $\chi_2|_{\Zp^{\times}}=\theta_2$;
\item[(3)] $\chi_1(p)\chi_2(p)p^{k-1}=1$;
\item[(4)] $\val(\chi_1(p))\le 0$, $\val(\chi_2(p))\le 0$;
\item[(5)] $\Hom_G(\kappa, E^0\otimes_{A'} k')\neq 0$.
\end{itemize}
Moreover, if we assume that $\kappa$ is not a subquotient of any principal series representation $\Indu{B}{G}{\psi_1\otimes\psi_2}$, 
such that  $(\psi_1|_{\Zp^{\times}}, \psi_2|_{\Zp^{\times}})=(\bar{\theta}_1, \bar{\theta}_2 \omega^{k-2})$ or 
$(\psi_1|_{\Zp^{\times}}, \psi_2|_{\Zp^{\times}})=(\bar{\theta}_2, \bar{\theta}_1 \omega^{k-2})$, then $\val(\chi_1(p))<0$ and $\val(\chi_2(p))<0$.
\end{thm}
 
\begin{proof} Corollaries \ref{cor1} and \ref{cor2} give us an extension $L'$ of $L$, an $L'$-principal series representation 
$\pi:=\Indu{B}{G}{\chi_1\otimes\chi_2|\centerdot|^{-1}}$, such that $\chi_1$, $\chi_2$ satisfy (2), (3) and an admissible unitary completion $E$ of 
$\pi\otimes \Sym^{k-2} (L')^2$, satisfying (5). For simplicity we assume that $L=L'$. Assume that $\chi_1=\chi_2$, then \cite[Cor 4.5]{except} says that 
there exists $x\in 1+\MM$, $x^2\neq 1$ and an admissible unitary completion $E_x$ of  
$$(\Indu{B}{G}{\chi_1\delta_x\otimes\chi_2\delta_{x^{-1}}|\centerdot|^{-1}})\otimes \Sym^{k-2} L^2,$$
such that $E^0_x \otimes_A k_L\cong E^0\otimes_A k_L$ as $G$-representations, where $\delta_x: \Qp^{\times}\rightarrow L^{\times}$ is an unramified 
character with $\delta_x(p)=x$. Hence, we may assume that $\chi_1\neq \chi_2$. The condition (4) follows from Lemma \ref{reform}. 
The last part follows from Lemma \ref{ordinary}.
\end{proof}

We use the results of Berger-Breuil \cite{bergerbreuil} and Berger \cite{berger} to transfer the statement of Theorem \ref{Qp} to the Galois side. 
Let $\GG_{\Qp}$ be the absolute Galois group of $\Qp$, and let $\II_{\Qp}$ be the inertia subgroup. We consider characters of $\Qp^{\times}$ as characters 
of $\GG_{\Qp}$ via class field theory, sending the geometric Frobenius to $p$, with this identification $\omega$ is the reduction modulo $p$ of 
the cyclotomic character. By a $2$-dimensional $L$-linear representation of $\GG_{\Qp}$ we 
mean a continuous group homomorphism $\GG_{\Qp}\rightarrow \GL_2(L)$, where $\GL_2(L)$ is equipped with the $p$-adic topology inherited from $L$.
Since $\GG_{\Qp}$ is compact, it will stabilize some $A$-lattice $T$ in $V$. Now $(T\otimes_A k_L)^{ss}$ does not depend on the choice of 
the lattice $T$, we denote this $k_L$-representation by $\overline{V}$. Given an integer $1\le s\le p$, we denote by $\ind \omega_2^s$ the unique 
$2$-dimensional $k_L$-representation $\rho$ of $\GG_{\Qp}$ such that $\det\rho=\omega^s$ and $\rho|_{\II_{\Qp}}\cong \omega_2^s \oplus \omega_2^{ps}$, where
$\omega_2$ is the fundamental character of level $2$, then $\ind \omega_2^s$ is absolutely irreducible, and any absolutely irreducible $2$-dimensional
$k_L$-representation of $\GG_{\Qp}$ is isomorphic to a twist of $\ind \omega_2^s$, for some $1\le s \le p$. 

 Recall that a representation $V$ of $\mathcal G_{\Qp}$ is crystabelline if it becomes crystalline after restriction to $\Gal(\Qbar/E)$, where $E$ 
is an abelian extension of $\Qp$. Absolutely irreducible $L$-linear $2$-dimensional crystabelline representations of $\mathcal G_{\Qp}$  
with Hodge-Tate weights $(0,k-1)$, ($k\ge 2$) can be parameterized by pairs of smooth characters $\alpha, \beta:\Qp^{\times} \rightarrow L^{\times}$, such that 
$-(k-1)< \val(\alpha(p))\le \val(\beta(p))<0$ and $\val(\alpha(p))+\val(\beta(p))=-(k-1)$, see \cite[Prop 2.4.5]{bergerbreuil} or 
\cite[\S 5.5]{colmez04b}. We denote by $V(\alpha, \beta)$ the unique crystabelline representation $V$, such that 
$D_{cris}(V)=D(\alpha, \beta)$, where $D(\alpha, \beta)$ is the  filtered 
admissible $L$-linear $(\varphi, \mathcal G_{\Qp})$-module defined in \cite[Def 2.4.4]{bergerbreuil}.

\begin{thm}\label{mainQp} Fix an integer $k\ge 2$ and smooth characters $\theta_1, \theta_2: \Zp^{\times}\rightarrow L^{\times}$. Let $M$ be a $K$-stable 
lattice in $\tau(\theta_1, \theta_2)\otimes \Sym^{k-2} L^2$. Suppose that $\sigma:=\Sym^r k_L^2 \otimes \det^a$ with $0\le r\le p-1$ 
and $0\le a<p-1$ is a subquotient of $M\otimes_A k_L$. Let $\rho$ be one of the following: 
\begin{itemize}
\item[(a)] $\rho=(\ind \omega_2^{r+1} )\otimes \omega^a$;
\item[(b)] if $(\omega^{r+1+a}\oplus \omega^a)|_{\II_{\Qp}}$ is not isomorphic to either $\bar{\theta}_1\oplus \bar{\theta_2}\omega^{k-1}$ or 
$\bar{\theta}_2\oplus \bar{\theta_1}\omega^{k-1}$ then let $\rho=\mu_{\lambda}\omega^{r+1+a} \oplus \mu_{\lambda^{-1}}\omega^a$, 
for any $\lambda\in k_L^{\times}$.
\end{itemize} 
Then there exists a finite extension $L'$ of $L$ and an absolutely irreducible crystabelline $L'$-representation 
$V:=V(\alpha, \beta)$ such that the following hold:
\begin{itemize}
\item[(1)] $\overline{V}\cong \rho$;
\item[(2)] $\alpha(p)\beta(p)p^{k-1}=1$;
\item[(3)] the Hodge-Tate weights of $V$ are $0$ and $k-1$;
\item[(4)] either ($\alpha|_{\Zp^{\times}}=\theta_1$ and $\beta|_{\Zp^{\times}}=\theta_2$) or 
($\alpha|_{\Zp^{\times}}=\theta_2$ and $\beta|_{\Zp^{\times}}=\theta_1$).
\end{itemize}
\end{thm}
\begin{proof} In case (a) set $\kappa:=\kappa(r,\omega^a)$. In case (b) if  
$(r,\lambda)=(0,\pm 1)$ then set $\kappa:=(\mu_{\pm 1}\omega^a )\circ \det$; if $(r,\lambda)=(p-1,\pm 1)$, then 
set $\kappa:=\Sp\otimes (\mu_{\pm 1}\omega^a) \circ \det$; otherwise set $\kappa:=\Indu{B}{G}{\mu_{\lambda^{-1}}\omega^a\otimes \mu_{\lambda}\omega^{a+r}}$. Lemma \ref{rational} implies that $\kappa$ can be realized 
over $k_L$. Moreover, it follows from Lemma \ref{socleK} that $\Hom_K(\sigma, \kappa)\neq 0$. 
Theorem \ref{Qp} gives an admissible unitary 
completion $E$ of $\Indu{B}{G}{\chi_1\otimes \chi_2|\centerdot|^{-1}}\otimes \Sym^{k-2} (L')^2$ with $\chi_1$, $\chi_2$ and $E$ satisfying 
conditions (1)-(5) of Theorem \ref{Qp}. We note that $\kappa$ is not a subquotient of $\Indu{B}{G}{\psi_1\otimes\psi_2}$, 
with $(\psi_1|_{\Zp^{\times}}, \psi_2|_{\Zp^{\times}})=(\bar{\theta}_1, \bar{\theta}_2 \omega^{k-2})$ and 
$(\psi_1|_{\Zp^{\times}}, \psi_2|_{\Zp^{\times}})=(\bar{\theta}_2, \bar{\theta}_1 \omega^{k-2})$. In case (a) 
this is automatic, since $\kappa$ is 
supersingular hence not a subquotient of any principal series, and in case (b) this follows from the assumption and 
Lemma \ref{addsoc}. In particular, we have $\val(\chi_1(p))<0$ 
and $\val(\chi_2(p))<0$. If $\val(\chi_1(p))\le \val(\chi_2(p))$ then set $\alpha:=\chi_1$ and $\beta:=\chi_2$, 
otherwise set 
$\alpha:=\chi_2$ and $\beta:=\chi_1$, so that $\val(\alpha(p))\le \val(\beta(p))$. 

If $\chi_1=\chi_2|\centerdot|$, then $\chi_1(p)=\chi_2(p)p^{-1}$, and so $\val(\chi_1(p))<\val (\chi_2(p))$. 
If $\chi_1=\chi_2|\centerdot|^{-1}$ then the representation $\Indu{B}{G}{\chi_1\otimes\chi_2|\centerdot|^{-1}}$ has $\chi_1\circ \det$ as 
a subobject. Hence, $\Sym^{k-2} (L')^2 \otimes \chi_1\circ \det$ admits a $G$-invariant lattice, which implies $k=2$. In particular, 
$\theta_1=\theta_2$, $r=0$  and $\omega^a=\bar{\theta}_1$. The assumption in (b), implies that we are in case (a), so that $\rho\cong 
(\ind \omega_2)\otimes \omega^a$. Since $k=2$ and $p>2$ it follows from \cite{bergerlizhu} that if  $V=V(\chi_1|\centerdot|^{-1}, \chi_1)$ then 
$\overline{V}\cong \rho$, see the example below. Assume that $\chi_1\neq \chi_2|\centerdot|^{\pm 1}$ 
then we have an isomorphism
$$\Indu{B}{G}{\chi_1\otimes\chi_2|\centerdot|^{-1}}\cong  \Indu{B}{G}{\chi_2\otimes\chi_1|\centerdot|^{-1}}.$$
So without loss of generality we may assume that $E$ is a unitary admissible completion of 
$\pi\otimes\Sym^{k-2} L^2$, $\pi:=\Indu{B}{G}{\alpha\otimes \beta|\centerdot|^{-1}}$ with 
\begin{itemize}
\item[(i)] $\alpha\neq \beta$;
\item[(ii)] $\alpha(p)\beta(p)p^{k-1}=1$;
\item[(iii)] $-(k-1)<\val(\alpha(p))\le \val(\beta(p))<0$.
\end{itemize} 
In this situation, Berger-Breuil have shown that the completion of $\pi\otimes \Sym^{k-2} L^2$ with respect to any finitely generated 
$A[G]$-lattice is topologically irreducible, \cite[Cor 5.3.2, 5.3.4]{bergerbreuil}. This implies that the completion $E$ is topologically 
irreducible and is isomorphic as a unitary $L$-Banach space representation of $G$ to the representation $B(V)$, with $V:=V(\alpha, \beta)$, 
defined in \cite[Def 4.2.4]{bergerbreuil}. In \cite{berger} Berger has shown that there are two possibilities:
\begin{itemize} 
\item[(A)] $E^0\otimes_A k\cong \kappa(s, b)$, with $0\le s \le p-1$ and $0\le b <p-1$, in which case $\overline{V}\cong (\ind \omega_2^{s+1}) 
\otimes \omega^b$;
\item[(B)] $(E^0\otimes_A k)^{ss}\cong (\Indu{B}{G}{\psi_1\otimes\psi_2\omega^{-1}})^{ss}\oplus 
(\Indu{B}{G}{\psi_2\otimes\psi_1\omega^{-1}})^{ss}$, in which case $\overline{V}\cong \psi_1\oplus \psi_2$.
\end{itemize}
Since, we know that $\Hom_G(\kappa, E^0\otimes_A k)\neq 0$, the result of Berger together with \cite[Thm 33, 34]{bl}, \cite[Cor 4.1.4]{breuil1} 
implies that $\overline{V}\cong \rho$.
\end{proof}   

\textbf{Example.} Assume that $\theta_1=\theta_2=\Eins$, so that $\tau(\theta_1, \theta_2)$ is the trivial representation of $K$. Fix an integer $k\ge 2$, 
and choose $\alpha_p, \beta_p\in \MM$, such that $\alpha_p \beta_p=p^{k-1}$, set $a_p:=\alpha_p+\beta_p$. We may assume that $\val(\alpha_p)\ge \val(\beta_p)$, 
define unramified characters $\alpha, \beta:\Qp^{\times}\rightarrow L^{\times}$, by $\alpha(p):=\alpha_p^{-1}$ and $\beta(p):=\beta_p^{-1}$. The representation
$V:=V(\alpha, \beta)$ is crystalline with Hodge-Tate weights $(0,k-1)$, and is isomorphic to the representation denoted by $V_{k,a_p}$ in \cite{berger}, 
\cite{bergerlizhu}. In \cite{bergerlizhu} and \cite{berger} the reduction $\overline{V}$ is computed when $2\le k\le 2p+1$, (see also \cite{bergerbreuil2}, 
the case $k=2p+1$ is an unpublished result of Breuil). We will illustrate the Theorem
in this case. Let $M$ be a $K$-stable lattice in $\Sym^{k-2} L^2$, with $2\le k\le 2p+1$. Let $\sigma:=\Sym^r k_L^2\otimes\det^a$ be an irreducible
subquotient of $M\otimes_A k$ and let $\rho$ be as in Theorem \ref{mainQp}. We will show that the assertion of Theorem \ref{mainQp} matches 
the computations of \cite{bergerlizhu}, \cite{berger}, that is there exists $a_p\in \MM$ such that $\overline{V}_{k,a_p}\cong \rho$. 
We note that the assumption in Theorem \ref{mainQp} (b) implies that we exclude the 
representations $\rho$, such that $\rho|_{\II_{\Qp}}\cong \Eins \oplus \omega^{k-1}$. 

If $2\le k\le p+1$ then $(M\otimes_A k_L)^{ss}\cong \Sym^{k-2} k_L^2$, 
and hence $\rho=\ind \omega_2^{k-1}$. Now it follows from \cite[Cor 4.1.3, Prop. 4.1.4]{bergerlizhu} that $\overline{V}\cong \rho$. 

If $k=p+2$ then \eqref{symp} below
gives   $(M\otimes_A k_L)^{ss}\cong \Sym^1 k_L^2 \oplus \Sym^{p-2} k_L^2 \otimes \det$, so $\rho$ is either $\ind \omega_2^2$, $(\ind \omega_2^{p-1})\otimes \omega
\cong \ind\omega^{2}$, \cite[Lem. 4.2.2]{breuil1} 
or $\mu_{\lambda} \omega^{p} \oplus \mu_{\lambda^{-1}} \omega= \mu_{\lambda}\omega \oplus \mu_{\lambda^{-1}}\omega$, for $\lambda\in k_L^{\times}$. If $0<\val(a_p)<1$ then
$\overline{V}\cong \ind \omega_2^2$, \cite[Thm 3.2.1]{berger}. If $\val(a_p)=1$ then $\overline{V}\cong \omega \mu_{\lambda}\oplus \omega\mu_{\lambda^{-1}}$, 
where $\lambda$ is a root of a polynomial $X^2-\overline{a_p/p} X+1$. Now $\alpha_p \beta_p=p^{p+1}$, $\val(\alpha_p)\ge \val(\beta_p)$,
$\val(\alpha_p +\beta_p)=1$ implies that $\val(\beta_p)=1$  and $\val(\alpha_p)>1$. So $a_p/p \equiv \beta_p/p\pmod{\MM}$. Choose any $u\in k_L^{\times}$, and 
let $[u]$ denote the Teichm\"uller lift. Replace $\alpha_p$ with $\alpha_p [u]^{-1}$ and $\beta_p$ with $\beta_p [u]$. Then $\overline{V}$ is isomorphic 
to $\omega\mu_{\lambda_u}\oplus \omega\mu_{\lambda_u^{-1}}$, where $\lambda_u$ is any root of the polynomial $X^2+ u \overline{a_p/p}X +1$. 
Given $\xi\in k_L^{\times}$, such that $\xi^2\neq -1$, let $u=-(\xi+\xi^{-1}) \overline{p/a_p}$, then $\lambda_u=\xi^{\pm 1}$.
If $\val(a_p)>1$ then $\overline{V}\cong \mu_{\sqrt{-1}}\omega \oplus \mu_{-\sqrt{-1}}\omega$, 
\cite[Cor. 4.1.3, Prop. 4.1.4]{bergerlizhu}.

If $p+3\le k\le 2p$,  then set $r_0:=(k-2)-(p-1)$, $r_1:=(k-2)-(p+1)$. We have $2\le r_0\le p-1$ and $0\le r_1\le p-3$. Then 
$$(M\otimes_A k_L)^{ss}\cong (\Sym^{r_1} k_L^2 \otimes \det )\oplus \Sym^{r_0} k_L^2 \oplus (\Sym^{p-1-r_0} k_L^2 \otimes \dt^{r_0})$$  
see \eqref{secondss}. Then $\rho$ is $\ind \omega_2^{r_0+1}\cong (\ind \omega_2^{p-r_0})\otimes \omega^{r_0}$, or $(\ind \omega_2^{r_1+1})\otimes \omega$, or
$\mu_{\lambda} \omega^{r_1+2} \oplus \mu_{\lambda^{-1}}\omega\cong \mu_{\lambda}\omega^{k-2}\oplus \mu_{\lambda^{-1}} \omega$, for some $\lambda\in k_L^{\times}$.
If $0<\val(a_p)<1$ then $\overline{V} \cong \ind \omega^{r_0+1}$, if $\val(a_p)>1$ then $\overline{V}\cong (\ind \omega_2^{r_1+1})\otimes \omega$, and 
if $\val(a_p)=1$ then $\overline{V}\cong \mu_{\lambda}\omega^{k-2}\oplus \mu_{\lambda^{-1}} \omega$, where $\lambda=(k-1)\overline{a_p/p}$, \cite[Thm 3.2.1]{berger}.
Again we see that every $\rho$ is isomorphic to some $\overline{V}$.   

If $k=2p+1$ then  \eqref{secondss} gives
$$(M\otimes_A k_L)^{ss} \cong (\Sym^{p-2} k_L^2\otimes \det)\oplus \Sym^1 k_L^2\oplus( \Sym^{p-2} k_L^2\otimes \det).$$ 
So the possibilities for $\rho$ are the same as in the case $k=p+2$, so that $\rho=\ind \omega_2^2$ or $\rho=\mu_{\lambda}\omega\oplus \mu_{\lambda^{-1}}\omega$, 
for  $\lambda\in k_L^{\times}$. If $\val(a_p^2+p)<3/2$ then $\overline{V}\cong \ind \omega_2^2$, if $\val(a_p^2+p)\ge 3/2$ then $\overline{V}\cong 
\omega\mu_{\lambda}\oplus \omega\mu_{\lambda^{-1}}$, where $\lambda$ is any root of $X^2+\overline{\frac{a_p^2 + p}{2p a_p}} X +1$, \cite[Thm 3.2.1]{berger}.
 We leave it as an exercise to 
an interested reader to check that given $\xi\in k_L^{\times}$ there exists $\alpha_p$, $\beta_p$, such that $\alpha_p \beta_p=p^{2p}$, 
$\val( (\alpha_p+\beta_p)^2+p)\ge 3/2$ and $\xi+\xi^{-1}= -\overline{\frac{(\alpha_p+\beta_p)^2+p}{2p(\alpha_p+\beta_p)}}$.

\begin{lem}\label{pokemon} Let $\theta_1, \theta_2:\Zp^{\times}\rightarrow L^{\times}$ be smooth characters and $k\ge 2$ an integer.
Let $M$ be a  $K$-stable lattice in $\tau(\theta_1, \theta_2)\otimes \Sym^{k-2}L^2$. We make the following assumptions:
\begin{itemize} 
\item[(a)] if $\theta_1=\theta_2$ then assume $k\ge p^2+1$;
\item[(b)] if $\theta_1\neq \theta_2$ and $\theta_1\theta_2^{-1}$ is trivial on $1+p\Zp$ then assume $k\ge p$.
\end{itemize}
Then every irreducible $k_L$-representation $\sigma$ of $K$ with the central character $\bar{\theta}_1\bar{\theta}_2 \omega^{k-2}$ is a subquotient of 
$M\otimes_A k$.
\end{lem}
\begin{proof} This is shown in the appendix.
\end{proof} 
  
\begin{cor}\label{cristabeline} Fix smooth characters $\theta_1, \theta_2:\Zp^{\times}\rightarrow L^{\times}$, and an integer $k\ge 2$, such that 
\begin{itemize} 
\item[(a)] if $\theta_1=\theta_2$ then assume $k\ge p^2+1$;
\item[(b)] if $\theta_1\neq \theta_2$ and $\theta_1\theta_2^{-1}$ is trivial on $1+p\Zp$ then assume $k\ge p$.
\end{itemize}
Let $\rho$ be a semisimple  $2$-dimensional $k_L$-representation of $\mathcal G_{\Qp}$, such that 
\begin{itemize} 
\item[(c)] $\det \rho|_{\mathcal I_{\Qp}}= \bar{\theta_1} \bar{\theta_2}\omega^{k-1}$;
\item[(d)] if $\rho$ is irreducible, then it is absolutely irreducible;
\item[(e)] $\rho|_{\II_{\Qp}}\not\cong \bar{\theta}_1\oplus \bar{\theta}_2 \omega^{k-1}$ and 
$\rho|_{\II_{\Qp}}\not \cong \bar{\theta}_2\oplus \bar{\theta}_1 \omega^{k-1}$.
\end{itemize}
Then there exists a finite extension $L'$ of $L$ and an absolutely 
irreducible $2$-dimensional crystabelline $L'$-representation $V:=V(\alpha, \beta)$ of $\mathcal G_{\Qp}$, such that
\begin{itemize} 
\item[(i)] $\overline{V}\cong \rho$;
\item[(ii)] Hodge-Tate weights of $V$ are  $(0,k-1)$;
\item[(iii)]  either ($\alpha|_{\Zp^{\times}}=\theta_1$ and $\beta|_{\Zp^{\times}}=\theta_2$) or ($\alpha|_{\Zp^{\times}}=\theta_2$ and $\beta|_{\Zp^{\times}}=\theta_1$).
\end{itemize}
\end{cor}
\begin{proof} After twisting by a character, we can get $\rho$ to be as in Theorem \ref{mainQp}. The assertion follows from Theorem \ref{mainQp} and Lemma 
\ref{pokemon}.
\end{proof}

\appendix
\section{Semi-simplification}
We prove Lemma \ref{pokemon}. To simplify the notation we set $n:=k-2$ we keep the assumption $p>2$ and notations of the previous section.
 Let $M$ be a $K$-invariant lattice in $\tau(\theta_1, \theta_2)\otimes \Sym^n L^2$. Since $\tau(\theta_1, \theta_2)\otimes \Sym^n L^2$ is 
a finite dimensional $L$-vector space, $(M\otimes_A k_L)^{ss}$ does not depend on the choice of $M$, see the proof of \cite[Thm 32]{serre}. 
Since $\theta_1$ and $\theta_2$ 
are smooth characters, $\theta_1(g)$ and $\theta_2(g)$ are roots of unity for all $g\in \Zp^{\times}$. Hence, $\theta_1$ and $\theta_2$ are $A$-valued. 
If $\delta: \Zp^{\times}\rightarrow L^{\times}$ is a smooth character then Lemma \ref{pokemon} holds for $\theta_1$ , $\theta_2$, $k$ if and only if 
it holds for $\theta_1\delta$, $\theta_2\delta$, $k$. In particular, if $c=0$ we may assume that $\theta_1=\theta_2=\Eins$, so that $\tau(\theta_1, \theta_2)$
is the trivial representation, and take $M:=\Sym^n A^2$, so that $M\otimes_A k_L\cong \Sym^n k_L^2$.
If $c>1$ let $M_1$ be the space of functions $f:K\rightarrow A$, such that $f(hg)=(\theta_1\otimes \theta_2)(h) f(g)$, for all $g\in K$, $h\in J_c$. Then $M_1$ 
is a $K$-invariant lattice in $\tau(\theta_1, \theta_2)$, so $M:=M_1\otimes_A \Sym^n A$ is a $K$-invariant lattice in $\tau(\theta_1, \theta_2)\otimes \Sym^n L^2$, 
and $M\otimes_A k_L\cong (\Indu{J_c}{K}{\bar{\theta}_1\otimes\bar{\theta}_2})\otimes \Sym^n k_L^2$.  Since $k_L$ contains $\mathbb F_p$ 
every irreducible 
representation is absolutely irreducible. Hence, as far as semi-simplification is concerned working over $k_L$ is 
the same as working over an algebraically closed field, see \cite[\S14.6]{serre}. 

We first look at the case $c=0$ and so $M\otimes_A k_L\cong \Sym^n k_L^2$. Since $K_1$ acts trivially on $\Sym^n k_L^2$, it is enough to compute 
the semi-simplification of $\Sym^n k_L^2$ as a representation of $\GL_2(\Fp)$. Recall that semi-simplification is 
determined by the Brauer character, which is a $\Qpbar$-valued function on $p$-regular conjugacy classes of $\GL_2(\Fp)$, \cite[\S18.2]{serre}.  
We have \cite[\S1]{diamond}: 
$$ \chi_n(\begin{pmatrix} \lambda & 0\\ 0 & \lambda\end{pmatrix})= (n+1) [\lambda]^n, \quad   \chi_n(\begin{pmatrix} \lambda & 0\\ 0 & \mu\end{pmatrix})= 
\frac{[\lambda]^{n+1}-[\mu]^{n+1}}{[\lambda]-[\mu]},$$
where $\lambda, \mu\in \Fp^{\times}$ and $\lambda\neq \mu$. Moreover, choose an embedding $\iota:\mathbb F^{\times}_{p^2}\rightarrow \GL_2(\Fp)$, suppose that 
$z\in \mathbb F_p^2 \setminus \Fp$ then $\chi_n(\iota(z))$ does not depend on $\iota$ and we have: 
$$\chi_n(z)= [z]^n \frac{[z]^{(p-1)(n+1)}- 1}{[z]^{p-1}-1}.$$

\begin{lem}\label{simp} Let $n\ge p+1$ be an integer then $\chi_n= \chi_{n-p-1}\det + \chi_r + \chi_{p-r-1}\det^r$, where $0\le r<p-1$ and $n\equiv r\pmod{p-1}$.
\end{lem}
\begin{proof} Let $\psi:=\chi_n-\chi_{n-p-1}\det$ then a calculation using the formulae above gives:
$$\psi(\begin{pmatrix} \lambda & 0\\ 0 & \lambda\end{pmatrix})= (p+1) [\lambda]^r, \quad \psi(\begin{pmatrix} \lambda & 0\\ 0 & \mu\end{pmatrix})= 
[\lambda]^r +[\mu]^r,\quad \psi(z)=0.$$
Let $B(\Fp)$ be the group of upper triangular matrices in $\GL_2(\Fp)$ and let $\chi: B(\Fp)\rightarrow \Fbar^{\times}$ be the character, given by 
$\chi(\bigl( \begin{smallmatrix} a & b \\0 & d\end{smallmatrix} \bigr ))= a^r$. It follows from \cite[\S1]{diamond} that $\psi$ is the Brauer character 
of $\Indu{B(\Fp)}{\GL_2(\Fp)}{\chi}$. Since 
$$  (\Indu{B(\Fp)}{\GL_2(\Fp)}{\chi})^{ss}\cong \Sym^r k_L^2 \oplus \Sym^{p-r-1} k_L^2 \otimes \dt^{r},$$
see eg. \cite[Lem 3.1.7, 4.1.3]{coeff}, we obtain the result.
\end{proof}
 If $0\le n\le p-1$ then $\Sym^n\Fbar^2$ is irreducible, if $n=p$ then 
\begin{equation}\label{symp}
(\Sym^p k_L^2)^{ss}\cong \Sym^1 k_L^2 \oplus \Sym^{p-2} k_L^2 \otimes \det,
\end{equation} 
see \cite[Lem.5.1.3]{breuil2}. Using Lemma \ref{simp}  and \eqref{symp} we may compute the semi-simplification. 
Let $m$ be the largest integer 
such that $n\ge (p+1)m$, for $0\le i\le m$ let $0\le r_i<p-1$ be the unique integer 
such that $n-(p+1)i \equiv r_i \pmod{p-1}$. If $n-(p+1)m=p$ then 
\begin{equation}\label{firstss}
(\Sym^n k_L^2)^{ss}\cong \bigoplus_{i=0}^m (\Sym^{r_i} k_L^2\otimes \dt^i \oplus  \Sym^{p-r_i-1} k_L^2\otimes \dt^{r_i+i}),
\end{equation}
otherwise, $(\Sym^n k_L^2)^{ss}$ is isomorphic to 
\begin{equation}\label{secondss}
 \Sym^{r_m} k_L^2 \otimes \dt^m \oplus \bigoplus_{i=0}^{m-1} 
(\Sym^{r_i} k_L^2\otimes \dt^i \oplus  \Sym^{p-r_i-1} k_L^2\otimes \dt^{r_i+i}).
\end{equation}

\begin{lem} Assume $p>2$ and let $n\ge p^2-1$ be an integer, let 
$\sigma$ be an irreducible $k_L$-representation of $K$, with central character $\omega^n$. Then 
$\sigma$ occurs as a subquotient of $\Sym^n k_L^2$.
\end{lem}
\begin{proof} We note that the central character of $\Sym^n k_L^2$ is $\omega^n$. Hence, every irreducible subquotient will also have a central character 
$\omega^n$. We have $\sigma\cong \Sym^r k_L^2 \otimes \det^a$ with $0\le r\le p-1$ and $0\le a <p-1$. The central character of $\sigma$ is equal to 
$\omega^{r+2a}$. The equality $\omega^{r+2a}=\omega^n$ implies that $2$ divides $n-r$.  Let $r_j$ be as above and note that $r_j\equiv r_0-2j\pmod{p-1}$. 
Since $p+1$ is even we get that $2$ divides
$r-r_0$. Let $0\le j<(p-1)/2$ be the unique  integer such that $r_0-2j \equiv r \pmod{p-1}$. Then $r=r_j=r_{j+(p-1)/2}$. The congruence, $2a\equiv n-r\equiv r_0-r 
\pmod{p-1}$ implies that either $a=j$ or $a=j+(p-1)/2$. Hence, either $\sigma\cong \Sym^{r_j}k_L^2\otimes \det^j$ or $\sigma\cong 
\Sym^{r_{j+(p-1)/2}} k_L^2 \otimes \det^{j+(p-1)/2}$. It follows from \eqref{firstss} and \eqref{secondss} that $\sigma$ is an irreducible subquotient of 
$\Sym^n k_L^2$. Note that the assumption $n\ge p^2-1$ implies that $m\ge p-1$.
\end{proof}
We look at the case $c\ge 1$, so that $M\otimes_A k_L\cong (\Indu{J_c}{K}{\bar{\theta}_1\otimes \bar{\theta}_2})\otimes \Sym^n k_L^2$.
 \begin{lem} If $\theta_1\theta_2^{-1}$ is trivial on $1+p\Zp$ then assume $n\ge p-2$. 
Then every irreducible $k_L$-representation $\sigma$ of $K$ with the central character $\bar{\theta}_1\bar{\theta}_2 \omega^{k-2}$ is a subquotient of 
$M\otimes_A k$.
\end{lem}
\begin{proof} We note that the central character of $M\otimes_A k$ is $\bar{\theta}_1\bar{\theta}_2 \omega^{n}$, 
hence every irreducible subquotient of $M\otimes_A k$ 
will have the central character $\bar{\theta}_1\bar{\theta}_2 \omega^{n}$. Set $Z_K:=Z \cap K$ then every irreducible $\sigma$ with the central character 
$\bar{\theta}_1\bar{\theta}_2 \omega^{k-2}$ will be a subquotient of 
\begin{equation}\label{decompeasy}
\Indu{Z_K I_1}{K}{\bar{\theta}_1\bar{\theta}_2 \omega^{n}}\cong \oplus_{i=0}^{p-2}\Indu{I}{K}{\omega^{n-i}\bar{\theta}_1\otimes \omega^i \bar{\theta}_2}. 
\end{equation}
Now 
\begin{equation}\label{yetanother}
M\otimes_A k\cong \Indu{J_c}{K}{((\bar{\theta}_1\otimes \bar{\theta}_2)
\otimes \Sym^{n} k_L^2)}.
\end{equation}
Since  $((\Sym^{n} k_L^2)|_{J_c})^{ss}\cong \oplus_{i=0}^{n} (\omega^{n-i}\otimes \omega^i)$, if $c=1$ and $n\ge p-2$ then $J_c=I$ and 
\eqref{decompeasy} and \eqref{yetanother} give the claim. Assume that $c>1$ then $(M\otimes_A k_L)^{ss}$ will contain 
$$(\Indu{J_2}{K}{\bar{\theta}_1\omega^{n}\otimes \bar{\theta}_2})^{ss}\cong 
(\Indu{I}{K}{((\bar{\theta}_1\omega^{n}\otimes \bar{\theta}_2})\otimes \Indu{J_2}{I}{\Eins}))^{ss}.$$
Now $(\Indu{J_2}{I}{\Eins})^{ss}\cong \oplus_{i=0}^{p-1} (\omega^{-i}\otimes \omega^{i})$. The claim follows from \eqref{decompeasy}.
\end{proof}

\end{document}